\documentclass[11pt]{amsart}
\usepackage[scale=.76]{geometry}

\usepackage{amsmath,amsthm}
\usepackage{amssymb}
\usepackage[shortlabels]{enumitem}
\usepackage[T1]{fontenc}
\usepackage[colorlinks=true,linkcolor=blue,citecolor=blue,urlcolor=blue]{hyperref}

\usepackage{graphicx} % For images and rotatebox
\usepackage{xcolor}   % For color support

\usepackage[utf8]{inputenc}

\usepackage{enumitem}
\usepackage{booktabs}

\usepackage[all]{xy}

\usepackage[initials, msc-links]{amsrefs}  % alphabetical by author, given names as initials;
\usepackage[capitalise]{cleveref}

\newtheorem{thm}{Theorem}[section]
\newtheorem{lem}[thm]{Lemma}
\newtheorem{prop}[thm]{Proposition}
\newtheorem{cor}[thm]{Corollary}
\newtheorem*{thmA}{Theorem A}
\newtheorem*{thmB}{Theorem B}
\newtheorem*{thmC}{Theorem C}
\newtheorem*{thmD}{Theorem D}
\newtheorem{conj}[thm]{Conjecture}

\theoremstyle{definition}
\newtheorem{defi}[thm]{Definition}
\newtheorem{exam}[thm]{Example}
\newtheorem{rem}[thm]{Remark}

\numberwithin{equation}{section}

\newcommand{\Z}{\mathbb Z}
\newcommand{\Q}{\mathbb Q}
\newcommand{\R}{\mathbb R}
\newcommand{\CC}{\mathbb C}
\newcommand{\bQ}{\overline{\mathbb Q}}
\renewcommand{\P}{\mathbb P}
\newcommand{\w}{\mathfrak q}
\newcommand{\lwh}{\mathfrak s}
\newcommand{\wh}{\mathcal S}
\newcommand{\x}{\mathbf x}
\newcommand{\y}{\mathbf y}
\newcommand{\cc}{\mathbf c}
\newcommand{\aaa}{\mathbf a}
\newcommand{\cO}{\mathcal O}
\newcommand{\nm}[1]{\lVert #1\rVert}
\DeclareMathOperator{\lcm}{lcm}
\DeclareMathOperator{\Proj}{Proj}
\DeclareMathOperator{\ind}{ind}
\DeclareMathOperator{\rad}{rad}

	\title[Roth's theorem on weighted projective lines]{Roth's theorem on weighted projective lines}
	
	
	
	\author{S. Salami}
	\address{Institute of Mathematics, Statistics and Computing, Rio de Janeiro State University, Maracanã, Rio de Janeiro, 20950-000, RJ, Brazil}
	\email{Sajad.salami@ime.uerj.br}

\begin{document}

	\begin{abstract}
		Let ${\mathbb P}^1_{\mathfrak q}$ be the weighted projective line with coprime weights ${\mathfrak q}=(q_0,q_1)$ and $L=q_0q_1$. The weighted Veronese map $\Phi\colon {\mathbb P}^1_{\mathfrak q}\to {\mathbb P}^1$ satisfies $h\circ\Phi=L\,{\mathfrak s}$. Pulling back Roth's theorem gives the optimal exponent $2L$ for the full rational orbit set. For points represented by coprime integral coordinates, counting suggests the smaller exponent $q_0+q_1$. We formulate this as a conjecture and prove an $S$-adic radical form under the $abc$ conjecture. For ${\mathfrak q}=(2,3)$ the resulting inequality at the target $[1:1]$ is the coprime Hall-type inequality, and Danilov--Elkies examples show that the exponent $5$ is sharp. We also prove the expected exponent unconditionally over characteristic-zero function fields for constant targets and polynomial approximants, and for almost every real target. A weighted $abc$ formulation is introduced; in signature $(3,2)$ finitely many coefficient pairs are equivalent to $abc$, and weighted Belyi maps transfer this input to arbitrary algebraic targets. Applications include a pullback form of Schmidt's subspace theorem, weighted Thue inequalities, and a counterexample to a naive weighted rescaling of Roth's index construction.
	\end{abstract}
	
	\maketitle

	\section{Introduction}\label{sec:intro}
	
	Roth's theorem \cite{Roth1955} states that an irrational algebraic number $\alpha$ admits only finitely many rational approximations $p/q$ with $|\alpha-p/q|<q^{-2-\epsilon}$, for every $\epsilon>0$. In terms of heights and Weil functions \cite{HS2000,Vojta2011,BG2006} it asserts, for $\beta\in\P^1(\bQ)$, that
	\[
	\lambda_{\beta,\infty}(y)\le(2+\epsilon)\,h(y)+O(1)\qquad\text{for all but finitely many }y\in\P^1(\Q),
	\]
	and the exponent $2$ is forced by Dirichlet's theorem, since $\P^1(\Q)$ has about $H^2$ points of height at most $H$. Schmidt's subspace theorem \cite{Schmidt1980,Schmidt1991} extends this to hyperplanes in $\P^n$.
	
	Weighted projective spaces $\P^n_{\w}$ occur naturally in moduli problems, for example for binary forms. Heights on weighted projective orbit sets were studied in \cite{BGS2020,SS2023}, and a weighted form of Vojta's conjecture appears in \cite{SSVojta}. Heights on weighted projective stacks have also been introduced in connection with counting problems \cite{Darda2021,ESZB2023,DY2022,NX2025}. In this paper we study Roth-type approximation on the weighted projective line. A basic point, essential throughout, is that over a non-algebraically closed field the weighted orbit set need not coincide with the set of rational points of the coarse variety. We retain the standard notation $\P^n_{\w}(K)$ for the orbit set and state this convention explicitly in Section~\ref{sec:heights}. The arithmetic distinction between the full rational orbit set and the subset represented by coprime integral coordinates is the source of the two different exponents considered below.
	
	Let $\w=(q_0,q_1)$ with $\gcd(q_0,q_1)=1$, put $L=q_0q_1$, and let
	\[
	\Phi\colon\P^1_{\w}\to\P^1,\qquad[x_0:x_1]_{\w}\mapsto[x_0^{q_1}:x_1^{q_0}]
	\]
	be the weighted Veronese map, an isomorphism of varieties. The weighted height $\lwh$ satisfies $h(\Phi(P))=L\,\lwh(P)$ exactly, place by place (Lemma~\ref{lem:veronese}). To a point $\alpha=[\alpha_0:\alpha_1]_{\w}$ we attach the form $F_\alpha=\alpha_1^{q_0}X_0^{q_1}-\alpha_0^{q_1}X_1^{q_0}$ of weighted degree $L$ and the weighted chordal distance $\delta_\infty(P,\alpha)$ (Definition~\ref{def:chordal}). For a pair $\x=(x_0,x_1)$ of integers we write $|\x|_{\w}=\max\{|x_0|^{1/q_0},|x_1|^{1/q_1}\}$.
	
	\begin{thmA}
		Let $\alpha\in\P^1_{\w}(\bQ)$ and $\epsilon>0$. Then
		\[
		-\log\delta_\infty(P,\alpha)\le(2+\epsilon)\,L\,\lwh(P)+O(1)
		\]
		for all $P\in\P^1_{\w}(\Q)$ outside $\Phi^{-1}(Z)$, where $Z$ is finite. If $\Phi(\alpha)$ is real and irrational, the constant $2L$ cannot be replaced by any smaller number. In particular, for every $\w$-reduced pair $\x\in\Z^2$ with $F_\alpha(\x)\neq0$ one has $|F_\alpha(\x)|\ge C\,|\x|_{\w}^{-(1+\epsilon)L}$ with $C=C(\alpha,\w,\epsilon)>0$.
	\end{thmA}
	
	This is Theorem~\ref{thm:wroth}, Proposition~\ref{prop:optimal} and Corollary~\ref{cor:rothQ}; the $S$-adic version over number fields is proved as well. Theorem~A is equivalent to Roth's theorem. The exponent $2L$ is sharp because $\Phi$ maps $\P^1_{\w}(\Q)$ onto $\P^1(\Q)$, so all rational numbers take part.
	
	The picture changes for \emph{integral} points, by which we mean weighted orbits admitting a representative $P=[x_0:x_1]_{\w}$ with coprime integers $x_0,x_1$. Under $\Phi$ these map to fractions whose numerator is a $q_1$-th power and whose denominator is a $q_0$-th power, up to sign. In the orbifold language of Darmon~\cite{DG1995}, their images are the corresponding integral points. There are $\asymp T^{q_0+q_1}$ such weighted orbits with $\lwh(P)\le\log T$, whereas the images of all rational weighted orbits in the same height range number $\asymp T^{2L}$. This counting argument gives a Dirichlet-type \emph{heuristic}, not a theorem, for the exponent $q_0+q_1$ (Remark~\ref{rem:heuristic}). We formulate the resulting inequality as Conjecture~\ref{conj:integral} and prove it under $abc$ for arbitrary algebraic targets.
	
	\begin{thmB}
		Assume the $abc$ conjecture. Let $\alpha\in\P^1_{\w}(\bQ)$ with $\Phi(\alpha)\notin\{[0:1],[1:0]\}$ and let $\epsilon>0$. Then
		\[
		-\log\delta_\infty(P,\alpha)\le(q_0+q_1+\epsilon)\,\lwh(P)+O(1)
		\]
		for all integral points $P$ with $\Phi(P)\neq\Phi(\alpha)$. For $\w=(2,3)$ and $\Phi(\alpha)=[1:1]$ the constant $q_0+q_1=5$ cannot be lowered, even though $2L=12$.
	\end{thmB}
	
	Theorem~\ref{thm:abcgeneral} proves more: a radical form (Conjecture~\ref{conj:radical}), and an $S$-adic version with several targets, valid for all rational points once $\lwh$ is replaced by the integral weighted height $\log|\x|_{\w}$ of a reduced representative. The proof combines Belyi's theorem with $abc$, as in Elkies' deduction of Mordell's conjecture from $abc$~\cite{Elkies1991}. The weights enter only through the inequality $\rad(u)\rad(v)\le|\x|_{\w}^{q_0+q_1}$ for the numerator and denominator of an integral point. For $\w=(2,3)$ and $\Phi(\alpha)=[1:1]$, the resulting inequality is, after the standard normalization of exponents, the coprime Hall-type inequality \cite{Hall1971}; the sharpness statement comes from the Danilov--Elkies construction \cite{Danilov1982,Elkies2000}. Thus a general unconditional proof of Conjecture~\ref{conj:integral} would contain a notoriously difficult special case (Remark~\ref{rem:barrier}). Unconditionally, Proposition~\ref{prop:functionfield} proves the expected exponent without $\epsilon$ over characteristic-zero function fields for constant targets and polynomial approximants, and Proposition~\ref{prop:metric} proves the real inequality for almost every target.
	
	Theorem~B uses $abc$ only for triples in which two terms are perfect powers. We isolate this as a weighted $abc$ conjecture $\mathrm{wABC}(p,q;A,B)$ (Definition~\ref{def:wabc}): the $abc$ inequality for $Ax^p+By^q=c$ in which $x^p$ and $y^q$ contribute $|x|$ and $|y|$ instead of their radicals. It follows from $abc$. For a single signature it suffices for rational targets, and more generally for targets $b$ with $b^n\in\Q$ (Proposition~\ref{prop:kummer}). For general algebraic targets the Belyi map in the proof of Theorem~B destroys the perfect-power structure. To restore it we introduce weighted Belyi maps: Belyi maps whose ramification indices are constant over $0$ and over $1$ (Definition~\ref{def:wbelyi}). The radical form of our conjecture is Conjecture~\ref{conj:radical}.
	
	\begin{thmD}
		\begin{enumerate}[(a)]
			\item For every finite set $S\subset\P^1(\bQ)$ there is a finite morphism $f\colon\P^1_{\Q}\to\P^1_{\Q}$ with the following properties: it is unramified outside $\{0,1,\infty\}$, it satisfies $S\subset f^{-1}(\infty)$, and all its ramification indices over $0$ equal $3$ and all over $1$ equal $2$.
			\item For each form $G$ as in Conjecture~\ref{conj:radical}, the weighted $abc$ conjecture of signature $(3,2)$ implies Conjecture~\ref{conj:radical} for $G$ and every coprime weight vector. Only finitely many coefficient pairs $(4e,-1)$ are needed, and they depend only on $G$.
			\item The $abc$ conjecture is equivalent to $\mathrm{wABC}(3,2;4e,-1)$ for $e\in\{1,3,9,27\}$.
			\item If Conjecture~\ref{conj:radical} holds for $\w=(2,3)$ and the four rational targets $[1:4e]$, $e\in\{1,3,9,27\}$, then it holds for all coprime weight vectors and all algebraic targets.
		\end{enumerate}
	\end{thmD}
	
	Part (a) is Theorem~\ref{thm:wbelyiexists}: one may take a Belyi map followed by $j/1728$, written in the Legendre parameter. Parts (b), (c) and (d) are Theorem~\ref{thm:wbelyi}, Proposition~\ref{prop:equiv} and Corollary~\ref{cor:universal}. Parts (b) and (c) together show that weighted Belyi maps do not lower the strength of the hypothesis. What they show precisely is that four radical inequalities in signature $(2,3)$ imply $abc$, and therefore imply the radical and integral conjectures for all coprime weight vectors through Theorem~\ref{thm:abcgeneral}.
	
	In Section~\ref{sec:subspace} we transfer Schmidt's subspace theorem to diagonal hypersurfaces $\sum_ic_ix_i^{L/q_i}=0$ in $\P^n_{\w}$. Section~\ref{sec:thue} treats weighted homogeneous binary forms, each of which has the shape $F=X^aY^bG(X^{q_1},Y^{q_0})$ with $G$ a binary form of degree $\kappa$ (Lemma~\ref{lem:structure}).
	
	\begin{thmC}
		Let $F$ be as above with coefficients in $\bQ$, and assume $G$ has no repeated factors. For every $\epsilon>0$ there is $C>0$ with $|F(x,y)|\ge C\,|(x,y)|_{\w}^{L(\kappa-2-\epsilon)}$ for all coprime $(x,y)\in\Z^2$ with $F(x,y)\neq0$. If $\kappa\ge3$, then $F(x,y)=m$ has only finitely many integral solutions for every $m\neq0$.
	\end{thmC}
	
	The hypothesis $\kappa\ge3$ cannot be relaxed to $\kappa\ge2$ in general (Remark~\ref{rem:sharpthue}). Under $abc$, finiteness holds as soon as $\kappa L>q_0+q_1$ (Corollary~\ref{cor:abcthue}). This covers the equations $x^p-y^q=m$ with coprime $p,q\ge2$, whose finiteness is known unconditionally by Siegel's theorem.
	
	Finally, Section~\ref{sec:direct} gives a counterexample to the most naive weighted rescaling of the index condition used in Roth's proof (Proposition~\ref{prop:index}). This rules out that direct rescaling, but not other anisotropic auxiliary-polynomial constructions.
	
	\subsection*{Conventions}
	Throughout, $k$ is a number field and $M_k$ its set of places. For $v\in M_k$ above $p\in M_{\Q}=\{\infty,2,3,5,\dots\}$ we put $\nm{x}_v=|N_{k_v/\Q_p}(x)|_p$, so that $\nm{x}_v=|x|_v^{n_v}$ with $n_v=[k_v:\Q_p]$, and the product formula $\prod_{v\in M_k}\nm{x}_v=1$ holds for $x\in k^*$. We set $N_v=n_v$ if $v$ is archimedean and $N_v=0$ otherwise. Whenever algebraic points enter at a place $v$, we fix an extension of $|\cdot|_v$ to $\bQ$ and extend $\nm{\cdot}_v=|\cdot|_v^{n_v}$ accordingly; at the archimedean place of $\Q$ this amounts to fixing an embedding $\bQ\subset\CC$. For $y=[y_0:\dots:y_n]\in\P^n(k)$ we write $h_k(y)=\sum_{v\in M_k}\log\max_i\nm{y_i}_v$ and $h=h_k/[k:\Q]$. The $abc$ conjecture \cite{Oesterle1988}, \cite[Ch.~12]{BG2006} asserts that for every $\epsilon>0$ there is $K_\epsilon$ such that $\max\{|a|,|b|,|c|\}\le K_\epsilon\rad(abc)^{1+\epsilon}$ for all coprime integers $a,b,c$ with $a+b=c$.
	
	\section{Weighted projective spaces and weighted heights}\label{sec:heights}
	
	\subsection{Weighted projective space}
	Let $\w=(q_0,\dots,q_n)$ be a vector of positive integers, called weights, and put $L=\lcm(q_0,\dots,q_n)$ and $r_i=L/q_i$. For a field~$K$, the group $K^*$ acts on $K^{n+1}\setminus\{0\}$ by
	\[
	\lambda\ast\x=(\lambda^{q_0}x_0,\dots,\lambda^{q_n}x_n),
	\]
	and we write $\P^n_{\w}(K)=(K^{n+1}\setminus\{0\})/K^*$, denoting the orbit of $\x$ by $[\x]_{\w}$. \emph{Throughout the paper this notation denotes the weighted orbit set, not the set of $K$-rational points of the coarse scheme.} This is the convention used in \cite{BGS2020,SS2023}. The weighted projective space as a variety is $\P^n_{\w}=\Proj k[x_0,\dots,x_n]$ with $\deg x_i=q_i$ \cite{Dolgachev1982,BR1986}. Each weighted orbit determines a $K$-point of the coarse variety, and the resulting map is bijective when $K$ is algebraically closed. Over a general field distinct weighted orbits can determine the same coarse point; see Lemma~\ref{lem:fibres}(c).
	
	The \emph{weighted Veronese map} $\Phi\colon\P^n_{\w}(K)\to\P^n(K)$, $\Phi([\x]_{\w})=[x_0^{r_0}:\dots:x_n^{r_n}]$, is well defined because $(\lambda^{q_i}x_i)^{r_i}=\lambda^Lx_i^{r_i}$. Geometrically, $\Phi$ is the finite surjective morphism $\P^n_{\w}\to\P^n$ given by the sections $x_i^{r_i}$ of the invertible sheaf $\cO_{\P^n_{\w}}(L)$. For $n=1$, a monomial $x_0^ix_1^j$ of weighted degree divisible by $L$ satisfies $r_0\mid i$ and $r_1\mid j$. Hence the $L$-th Veronese subring of $k[x_0,x_1]$ is $k[x_0^{r_0},x_1^{r_1}]$, and $\Phi$ is an isomorphism $\P^1_{\w}\cong\P^1$.
	
	\subsection{Weighted heights}
	Let $K$ be a number field. For $\x\in K^{n+1}\setminus\{0\}$ and $v\in M_K$ put $\nm{\x}_{v,\w}=\max_i\nm{x_i}_v^{1/q_i}$. Since $\nm{\lambda\ast\x}_{v,\w}=\nm{\lambda}_v\nm{\x}_{v,\w}$, the product formula shows that
	\[
	\lwh_K(P)=\sum_{v\in M_K}\log\nm{\x}_{v,\w},\qquad P=[\x]_{\w}\in\P^n_{\w}(K),
	\]
	does not depend on $\x$. We write $\wh_K=\exp\lwh_K$ and $\lwh=\lwh_K/[K:\Q]$ for the absolute weighted height, which does not depend on $K$. This is the weighted height of \cite{BGS2020,SS2023}.
	
	\begin{lem}\label{lem:veronese}
		For $\x\in K^{n+1}\setminus\{0\}$ and $v\in M_K$ we have $\max_i\nm{x_i^{r_i}}_v=\nm{\x}_{v,\w}^L$. Consequently $h_K(\Phi(P))=L\,\lwh_K(P)$ and $h(\Phi(P))=L\,\lwh(P)$ for all $P\in\P^n_{\w}(K)$; in particular $\lwh\ge0$.
	\end{lem}
	
	\begin{proof}
		Since $q_ir_i=L$, we have $\nm{x_i^{r_i}}_v=(\nm{x_i}_v^{1/q_i})^L$, and $t\mapsto t^L$ is increasing on $[0,\infty)$. Summing over $v$ gives the rest.
	\end{proof}
	
	Thus $\lwh$ is $1/L$ times the height attached to $\cO_{\P^n_{\w}}(L)=\Phi^*\cO_{\P^n}(1)$ with the pulled-back standard metric.
	
	\begin{lem}\label{lem:fibres}
		Let $n=1$, assume $\gcd(q_0,q_1)=1$, and let $K$ be a field.
		\begin{enumerate}[(a)]
			\item $\Phi\colon\P^1_{\w}(K)\to\P^1(K)$ is surjective.
			\item If $x_0x_1\neq0$ and $\Phi([\x]_{\w})=\Phi([\y]_{\w})$, then $[\x]_{\w}=[\y]_{\w}$.
			\item $\Phi^{-1}([1:0])=\{[a:0]_{\w}:a\in K^*\}$, and $[a:0]_{\w}=[a':0]_{\w}$ if and only if $a'/a\in K^{*q_0}$; similarly for $[0:1]$. If $K$ is a number field, all these points have weighted height $0$.
		\end{enumerate}
	\end{lem}
	
	\begin{proof}
		(a) The points $[1:0]$, $[0:1]$ are images of $[1:0]_{\w}$, $[0:1]_{\w}$. For $t\in K^*$ choose $u,v\in\Z$ with $uq_1-vq_0=1$; then $\Phi([t^u:t^v]_{\w})=[t^{uq_1}:t^{vq_0}]=[t:1]$.
		(b) Write $y_i=t_ix_i$; the hypothesis gives $t_0^{q_1}=t_1^{q_0}$. With $a,b\in\Z$ such that $aq_0+bq_1=1$ and $\lambda=t_0^at_1^b$ we get $\lambda^{q_0}=t_0^{aq_0+bq_1}=t_0$ and $\lambda^{q_1}=t_1^{aq_0+bq_1}=t_1$, so $\y=\lambda\ast\x$.
		(c) The first assertions are immediate, and $\lwh_K([a:0]_{\w})=q_0^{-1}\sum_v\log\nm{a}_v=0$.
	\end{proof}
	
	\begin{rem}\label{rem:northcott}
		If $q_0\ge2$, then $\P^1_{\w}(\Q)$ contains the infinitely many points $[a:0]_{\w}$, $a\in\Q^*/\Q^{*q_0}$, all of height $0$. They define a single point of the variety $\P^1_{\w}$. Northcott's theorem fails for $\P^1_{\w}(\Q)$ but holds for its image under $\Phi$, which is why exceptional sets below have the form $\Phi^{-1}(Z)$ with $Z$ finite.
	\end{rem}
	
	\subsection{Integral representatives}
	Let $k=\Q$. A vector $\x\in\Z^{n+1}\setminus\{0\}$ is \emph{$\w$-reduced} if no prime $p$ satisfies $p^{q_i}\mid x_i$ for all $i$; we write $|\x|_{\w}=\max_i|x_i|^{1/q_i}$ and $v_p$ for the $p$-adic valuation.
	
	\begin{lem}\label{lem:integral}
		Every $P\in\P^n_{\w}(\Q)$ has a $\w$-reduced representative. For every representative $\x\in\Z^{n+1}$ of $P$,
		\[
		\lwh(P)=\log|\x|_{\w}-\sum_p\mu_p(\x)\log p,\qquad \mu_p(\x)=\min_i\frac{v_p(x_i)}{q_i}.
		\]
		If $\x$ is $\w$-reduced and $g=\gcd(x_0,\dots,x_n)$, then
		\[
		\frac{1}{\max_iq_i}\log\rad(g)\ \le\ \log|\x|_{\w}-\lwh(P)\ \le\ \log\rad(g),
		\]
		and the right-hand inequality is strict if $g>1$. In particular $\lwh(P)=\log|\x|_{\w}$ if and only if $g=1$.
	\end{lem}
	
	\begin{proof}
		If $b$ is a common denominator of a rational representative $\x$, then $b\ast\x\in\Z^{n+1}$; while some prime satisfies $p^{q_i}\mid x_i$ for all $i$, replace $\x$ by $p^{-1}\ast\x$, which lowers $\sum_i|x_i|$. The formula holds because $\log\max_i|x_i|_p^{1/q_i}=-\mu_p(\x)\log p$. If $\x$ is $\w$-reduced, then $\mu_p(\x)<1$ for every $p$. Moreover $\mu_p(\x)=0$ if $p\nmid g$, and $\mu_p(\x)\ge1/\max_iq_i$ if $p\mid g$.
	\end{proof}
	
	\begin{exam}\label{ex:height}
		For $\w=(2,3)$ and $\x=(2,4)$ we get $\lwh(P)=\tfrac16\log2$ and $\log|\x|_{\w}=\tfrac23\log2$; indeed $\Phi(P)=[1:2]$ and $h(\Phi(P))=\log2=6\,\lwh(P)$. So $H(\Phi(P))\neq|\x|_{\w}^L$ in general, even for $\w$-reduced $\x$.
	\end{exam}
	
	\section{Diagonal forms and weighted Weil functions}\label{sec:weil}
	
	\begin{defi}\label{def:diagonal}
		For $\cc\in\bQ^{n+1}\setminus\{0\}$ the \emph{diagonal form} of weighted degree $L$ is $F_{\cc}=\sum_ic_ix_i^{r_i}$, and $H_{\cc}=\{F_{\cc}=0\}\subset\P^n_{\w}$ is the corresponding \emph{diagonal hypersurface}.
	\end{defi}
	
	We have $F_{\cc}(\lambda\ast\x)=\lambda^LF_{\cc}(\x)$ and $F_{\cc}(\x)=\ell_{\cc}(x_0^{r_0},\dots,x_n^{r_n})$ with $\ell_{\cc}(\y)=\sum_ic_iy_i$, so $H_{\cc}=\Phi^{-1}(\Lambda_{\cc})$ for the hyperplane $\Lambda_{\cc}=\{\ell_{\cc}=0\}$. When the weights are not all equal, diagonal forms span a proper subspace of the forms of weighted degree $L$; for $\w=(1,1,2)$, for example, $x_0x_1$ is not diagonal.
	
	Let $K$ contain the coordinates of $\cc$ and put $\nm{\cc}_v=\max_i\nm{c_i}_v$. For $v\in M_K$, $P=[\x]_{\w}\in\P^n_{\w}(K)\setminus H_{\cc}$ and $y=[\y]\in\P^n(K)\setminus\Lambda_{\cc}$ define
	\[
	\zeta_{\cc,v}(P)=-\log\frac{\nm{F_{\cc}(\x)}_v}{\nm{\cc}_v\,\nm{\x}_{v,\w}^L},\qquad \lambda_{\cc,v}(y)=-\log\frac{\nm{\ell_{\cc}(\y)}_v}{\nm{\cc}_v\max_i\nm{y_i}_v}.
	\]
	The function $\lambda_{\cc,\cdot}$ is the standard Weil function of $\Lambda_{\cc}$ \cite[Def.~7.9]{Vojta2011}, and both expressions are independent of representatives.
	
	\begin{lem}\label{lem:weilcompare}
		$\zeta_{\cc,v}(P)=\lambda_{\cc,v}(\Phi(P))$ and $\zeta_{\cc,v}(P)\ge-N_v\log(n+1)$. Hence $\zeta_{\cc,\cdot}$ is a Weil function for the Cartier divisor $\Phi^*\Lambda_{\cc}$ on $\P^n_{\w}$.
	\end{lem}
	
	\begin{proof}
		The identity follows from Lemma~\ref{lem:veronese} and the bound from
		\[
		\nm{\ell_{\cc}(\y)}_v\le(n+1)^{N_v}\nm{\cc}_v\max_i\nm{y_i}_v;
		\]
		the last assertion follows from functoriality of Weil functions \cite[Thm.~8.8]{Vojta2011}.
	\end{proof}
	
	\begin{prop}[Weighted first main theorem]\label{prop:fmt}
		Let $\cc\in k^{n+1}\setminus\{0\}$, $P\in\P^n_{\w}(k)\setminus H_{\cc}$ and $S\subset M_k$ finite. With $m_S(\cc,P)=\sum_{v\in S}\zeta_{\cc,v}(P)$ and $N_S(\cc,P)=\sum_{v\notin S}\zeta_{\cc,v}(P)$,
		\[
		m_S(\cc,P)+N_S(\cc,P)=L\,\lwh_k(P)+h_k([\cc]).
		\]
	\end{prop}
	
	\begin{proof}
		Sum over $v\in M_k$ and apply the product formula to $F_{\cc}(\x)\in k^*$.
	\end{proof}
	
	\begin{defi}\label{def:chordal}
		Let $n=1$ and $\alpha=[\alpha_0:\alpha_1]_{\w}\in\P^1_{\w}(\bQ)$ with representative $\aaa=(\alpha_0,\alpha_1)$. Put
		\[
		F_\alpha(X_0,X_1)=\alpha_1^{r_1}X_0^{r_0}-\alpha_0^{r_0}X_1^{r_1},\qquad \delta_v(P,\alpha)=\frac{\nm{F_\alpha(\x)}_v}{\nm{\x}_{v,\w}^L\,\nm{\aaa}_{v,\w}^L}\quad(P=[\x]_{\w}).
		\]
		We call $\delta_v$ the \emph{weighted chordal distance}.
	\end{defi}
	
	For $y,\beta\in\P^1(\bQ)$ let $\delta_v(y,\beta)=\nm{y_0\beta_1-y_1\beta_0}_v/(\max_i\nm{y_i}_v\max_i\nm{\beta_i}_v)$ be the chordal distance, so that $\lambda_{\beta,v}=-\log\delta_v(\cdot,\beta)$ is the Weil function of the point $\beta$.
	
	\begin{lem}\label{lem:chordal}
		$\delta_v(P,\alpha)$ is independent of representatives, and $\delta_v(P,\alpha)=\delta_v(\Phi(P),\Phi(\alpha))$. In particular $\delta_v(P,\alpha)=0$ if and only if $\Phi(P)=\Phi(\alpha)$, and $-\log\delta_v(P,\alpha)=\zeta_{\cc,v}(P)$ for $\cc=(\alpha_1^{r_1},-\alpha_0^{r_0})$.
	\end{lem}
	
	\begin{proof}
		With $\y=(x_0^{r_0},x_1^{r_1})$ and $\boldsymbol\beta=(\alpha_0^{r_0},\alpha_1^{r_1})$ one has $F_\alpha(\x)=y_0\beta_1-y_1\beta_0$, and Lemma~\ref{lem:veronese} identifies the denominators; note $\nm{\cc}_v=\nm{\aaa}_{v,\w}^L$.
	\end{proof}
	
	\section{Roth's theorem for rational points}\label{sec:roth}
	
	\begin{thm}[Roth]\label{thm:roth}
		Let $S\subset M_k$ be finite and $\beta_v\in\P^1(\bQ)$ for $v\in S$. For every $\epsilon>0$ and $c\in\R$,
		\[
		\sum_{v\in S}-\log\delta_v(y,\beta_v)\le(2+\epsilon)\,h_k(y)+c
		\]
		for all but finitely many $y\in\P^1(k)$.
	\end{thm}
	
	For $\beta_v=[b_v:1]$ and $y=[x:1]$ one has $-\log\delta_v(y,\beta_v)=\log^+\nm{x-b_v}_v^{-1}+O(1)$, so in this case the statement is \cite[Thm.~3.3 and (3.4)]{Vojta2011}; see also \cite[Thm.~D.2.1]{HS2000} and \cite[Ch.~6]{BG2006}. The general case follows by a $k$-rational change of coordinates moving every $\beta_v$ away from $[1:0]$, which changes heights and Weil functions by bounded amounts \cite[Thm.~8.8]{Vojta2011}. Since $-\log\delta_v\ge-N_v\log2$, the set $S$ may be enlarged at will.
	
	\begin{thm}[Weighted Roth theorem]\label{thm:wroth}
		Let $\w=(q_0,q_1)$, $L=\lcm(q_0,q_1)$, let $S\subset M_k$ be finite and $\alpha_v\in\P^1_{\w}(\bQ)$ for $v\in S$. For every $\epsilon>0$ and $c\in\R$ there is a finite set $Z\subset\P^1(k)$ with
		\[
		\begin{aligned}
			\sum_{v\in S}-\log\delta_v(P,\alpha_v)
			&\le(2+\epsilon)\,L\,\lwh_k(P)+c,
		\end{aligned}
		\]
		for all $P\in\P^1_{\w}(k)$ with $\Phi(P)\notin Z$.
	\end{thm}
	
	\begin{proof}
		By Lemmas~\ref{lem:chordal} and~\ref{lem:veronese} this is Theorem~\ref{thm:roth} for $y=\Phi(P)$ and $\beta_v=\Phi(\alpha_v)$.
	\end{proof}
	
	For $\w=(1,1)$ Theorem~\ref{thm:wroth} is Theorem~\ref{thm:roth}, so the two are equivalent. The content of the weighted statement is that it holds for the adelic height $\lwh$, for a distance of weighted degree $L$, and with exceptional sets finite modulo the fibres of $\Phi$. As in Roth's theorem, $Z$ is ineffective \cite[Rem.~D.7.2]{HS2000}.
	
	\begin{cor}\label{cor:rothQ}
		Let $\alpha\in\P^1_{\w}(\bQ)$ and $\epsilon>0$. There is $C=C(\alpha,\w,\epsilon)>0$ such that $|F_\alpha(\x)|\ge C|\x|_{\w}^{-(1+\epsilon)L}$ for every $\w$-reduced $\x\in\Z^2$ with $F_\alpha(\x)\neq0$.
	\end{cor}
	
	\begin{proof}
		Write $P=[\x]_{\w}$, so that $\delta_\infty(P,\alpha)=|F_\alpha(\x)|/(|\x|_{\w}^L|\aaa|_{\w}^L)$, and let $Z$ be given by Theorem~\ref{thm:wroth} with $k=\Q$, $S=\{\infty\}$, $c=0$. If $\Phi(P)\notin Z$, then Lemma~\ref{lem:integral} gives $-\log\delta_\infty(P,\alpha)\le(2+\epsilon)L\,\lwh(P)\le(2+\epsilon)L\log|\x|_{\w}$, whence $|F_\alpha(\x)|\ge|\aaa|_{\w}^L|\x|_{\w}^{-(1+\epsilon)L}$. If $\Phi(P)=z\in Z$ with $z\neq\Phi(\alpha)$, then $\delta_\infty(P,\alpha)=\delta_\infty(z,\Phi(\alpha))=:\delta_z>0$, and $|F_\alpha(\x)|=\delta_z|\aaa|_{\w}^L|\x|_{\w}^L\ge\delta_z|\aaa|_{\w}^L|\x|_{\w}^{-(1+\epsilon)L}$ because $|\x|_{\w}\ge1$.
	\end{proof}
	
	\begin{prop}\label{prop:optimal}
		Assume $\gcd(q_0,q_1)=1$, and let $\alpha\in\P^1_{\w}(\bQ)$ with $\Phi(\alpha)=[b:1]$, $b\in\R\setminus\Q$. There are infinitely many $P\in\P^1_{\w}(\Q)$, with distinct images under $\Phi$, such that $-\log\delta_\infty(P,\alpha)\ge2L\,\lwh(P)-\log2$.
	\end{prop}
	
	\begin{proof}
		By Dirichlet's theorem \cite[Prop.~D.1.1]{HS2000} there are infinitely many coprime $(u,v)$ with $v\ge1$ and $|u/v-b|<v^{-2}$. Put $y=[u:v]$ and $M=\max\{|u|,|v|\}$, so $h(y)=\log M$ and $M\le(|b|+1)v$. Then
		\[
		\begin{aligned}
			\delta_\infty(y,[b:1])
			&=\frac{|u-bv|}{M\max\{|b|,1\}}
			<\frac{1}{vM\max\{|b|,1\}}\\
			&\le\frac{|b|+1}{\max\{|b|,1\}}M^{-2}
			\le2M^{-2}.
		\end{aligned}
		\]
		Choose $P$ with $\Phi(P)=y$ (Lemma~\ref{lem:fibres}(a)); then $h(y)=L\,\lwh(P)$ and $\delta_\infty(P,\alpha)=\delta_\infty(y,[b:1])$.
	\end{proof}
	
	\begin{rem}\label{rem:nonintegral}
		The points in Proposition~\ref{prop:optimal} are in general not integral. For instance, if $\w=(1,2)$ and $u$ is squarefree, the $\w$-reduced representative of $\Phi^{-1}([u:v])$ is $(u,uv)$, whose coordinates have greatest common divisor $|u|$, and Lemma~\ref{lem:integral} gives $\log|\x|_{\w}-\lwh(P)=\tfrac12\log|u|$. This is the source of the gap between Theorem~\ref{thm:wroth} and the behaviour of integral points studied next.
	\end{rem}
	
	\section{Integral points and the exponent \texorpdfstring{$q_0+q_1$}{q0+q1}}\label{sec:integral}
	
	In this section $\gcd(q_0,q_1)=1$, so $L=q_0q_1$, $r_0=q_1$ and $r_1=q_0$. We write $\P^1_{\w}(\Z)$ for the set of $P\in\P^1_{\w}(\Q)$ having a representative $\x\in\Z^2$ with $\gcd(x_0,x_1)=1$, and call such $P$ \emph{integral}.
	
	\subsection{Integral points and the integral weighted height}
	
	\begin{lem}\label{lem:integralpoints}
		\begin{enumerate}[(a)]
			\item Every $P\in\P^1_{\w}(\Q)$ has a $\w$-reduced representative $\x\in\Z^2$, unique up to $\x\mapsto(-1)\ast\x$. In particular $\wh_{\Z}(P):=|\x|_{\w}$ is well defined, and $\log\wh_{\Z}(P)\ge\lwh(P)$, with equality if and only if $P$ is integral.
			\item For $P\in\P^1_{\w}(\Z)$ with coprime representative $\x$, we have $\Phi(P)=[x_0^{q_1}:x_1^{q_0}]$ with $\gcd(x_0^{q_1},x_1^{q_0})=1$. The map $\Phi$ is at most two-to-one on $\P^1_{\w}(\Z)$, and its image consists of the $[u:v]\in\P^1(\Q)$ with $u,v$ coprime integers such that $\pm u$ is a $q_1$-th power and $\pm v$ is a $q_0$-th power.
			\item As $T\to\infty$, the number of $P\in\P^1_{\w}(\Z)$ with $\lwh(P)\le\log T$ is $\asymp T^{q_0+q_1}$, whereas the number of points $\Phi(P)$ with $P\in\P^1_{\w}(\Q)$ and $\lwh(P)\le\log T$ is $\asymp T^{2L}$.
		\end{enumerate}
	\end{lem}
	
	\begin{proof}
		(a) Existence is Lemma~\ref{lem:integral}. If $\x,\y$ are $\w$-reduced and $\y=\lambda\ast\x$ with $\lambda=a/c$ in lowest terms, then $c^{q_0}\mid x_0$ and $c^{q_1}\mid x_1$, so $c=1$ by reducedness, and $a=\pm1$ by symmetry. The comparison of heights is Lemma~\ref{lem:integral}.
		(b) The first claim is clear. Injectivity away from $[1:0],[0:1]$ is Lemma~\ref{lem:fibres}(b); each of these two points has at most the two integral preimages $[\pm1:0]_{\w}$, respectively $[0:\pm1]_{\w}$. Signs can be absorbed because $q_0$ or $q_1$ is odd.
		(c) Up to a factor at most $2$, the first count is the number of coprime pairs with $|x_0|\le T^{q_0}$ and $|x_1|\le T^{q_1}$, which is $\asymp T^{q_0+q_1}$ by M\"obius inversion. The second is the number of $y\in\P^1(\Q)$ with $H(y)\le T^L$, which is $\asymp T^{2L}$.
	\end{proof}
	
	In the language of Darmon \cite{DG1995}, the points $\Phi(P)$ with $P$ integral are the integral points of the orbifold $(\P^1,(1-\tfrac1{q_1})[0]+(1-\tfrac1{q_0})[\infty])$. Equivalently, the coprime-coordinate condition describes the integral weighted orbits associated with the corresponding stacky curve; we use this interpretation only as motivation and continue to work with the explicit orbit-set definition above. The number $(q_0+q_1)/L=\tfrac1{q_0}+\tfrac1{q_1}$ is minus the degree of $K_{\P^1}+(1-\tfrac1{q_1})[0]+(1-\tfrac1{q_0})[\infty]$. We call $\wh_{\Z}$ the \emph{integral weighted height}; heights of this form arise for weighted projective stacks \cite{Darda2021,ESZB2023}.
	
	\begin{rem}\label{rem:heuristic}
		Fix $\alpha$ with $\Phi(\alpha)=[b:1]$, $b\in\R^*$. By Lemma~\ref{lem:integralpoints}(c), the integral points with $\lwh(P)\approx\log T$ and $\Phi(P)$ near $b$ give about $T^{q_0+q_1}$ fractions of height about $T^L$. If these were randomly distributed, the closest to $b$ would be at distance about $T^{-(q_0+q_1)}$, that is, $-\log\delta_\infty(P,\alpha)\approx(q_0+q_1)\lwh(P)$. The same heuristic for all rational points gives the exponent $2L$ of Theorem~\ref{thm:wroth}. Since $q_0+q_1<2L$ unless $\w=(1,1)$, Theorem~\ref{thm:wroth} should be far from optimal on integral points.
	\end{rem}
	
	\begin{conj}\label{conj:integral}
		Let $\alpha\in\P^1_{\w}(\bQ)$ with $\Phi(\alpha)\notin\{[0:1],[1:0]\}$ and $\epsilon>0$. Then
		\[
		-\log\delta_\infty(P,\alpha)\le(q_0+q_1+\epsilon)\,\lwh(P)+O(1)
		\]
		for all $P\in\P^1_{\w}(\Z)$ with $\Phi(P)\neq\Phi(\alpha)$. Equivalently, $|F_\alpha(\x)|\ge C\,|\x|_{\w}^{L-q_0-q_1-\epsilon}$ for all coprime $\x\in\Z^2$ with $F_\alpha(\x)\neq0$.
	\end{conj}
	
	The two formulations agree because $-\log\delta_\infty(P,\alpha)=-\log|F_\alpha(\x)|+L\log|\x|_{\w}+O(1)$ and $\lwh(P)=\log|\x|_{\w}$ on integral points. The hypothesis on $\Phi(\alpha)$ is necessary: for $\alpha=[0:1]_{\w}$ and $P=[1:x_1]_{\w}$ one finds $-\log\delta_\infty(P,\alpha)=L\,\lwh(P)$, and $L>q_0+q_1$ as soon as $q_0,q_1\ge2$. For $\w=(1,1)$ the conjecture is Roth's theorem. If $\Phi(\alpha)$ is rational and $\min\{q_0,q_1\}=1$ it is trivial, because then $F_\alpha(\x)$ is a fixed multiple of a nonzero integer and $L=\max\{q_0,q_1\}<q_0+q_1$. For the arithmetic applications we also need a radical form of the conjecture. For a squarefree binary form $G\in\Z[U,V]$ we write $Z(G)\subset\P^1(\bQ)$ for its set of zeros.
	
	\begin{conj}[Radical form]\label{conj:radical}
		Let $G\in\Z[U,V]$ be a squarefree binary form of degree $d\ge1$ with $G(0,1)G(1,0)\neq0$, and let $\epsilon>0$. There is $C=C(\w,G,\epsilon)>0$ such that
		\[
		\rad\bigl(G(x_0^{q_1},x_1^{q_0})\bigr)\ge C\,|\x|_{\w}^{dL-q_0-q_1-\epsilon}
		\]
		for all coprime $\x\in\Z^2$ with $x_0x_1\,G(x_0^{q_1},x_1^{q_0})\neq0$.
	\end{conj}
	
	For $\w=(1,1)$ this is the radical form of the Thue--Siegel--Roth inequality for binary forms, which follows from $abc$ (Lemma~\ref{lem:radical} below; see also \cite{Granville1998}). The next lemma converts lower bounds for radicals into lower bounds for $S$-adic distances.
	
	\begin{lem}\label{lem:deduction}
		Let $G\in\Z[U,V]$ be squarefree of degree $d$ with $G(0,1)G(1,0)\neq0$, let $S\subset M_{\Q}$ be finite with $\infty\in S$, and for $v\in S$ let $A_v\subset\P^1_{\w}(\bQ)$ be a finite set on which $\Phi$ is injective, with $\Phi(A_v)\subset Z(G)$. There is $C>0$ such that
		\[
		\prod_{v\in S}\prod_{\alpha\in A_v}\delta_v(P,\alpha)\ \ge\ C\,\rad\bigl(G(u,v)\bigr)\,H(\Phi(P))^{-d}
		\]
		for all $P\in\P^1_{\w}(\Q)$ with $\Phi(P)=[u:v]$, $u,v$ coprime integers and $G(u,v)\neq0$. Consequently, if Conjecture~\ref{conj:radical} holds for $\w$ and $G$, then for every $\epsilon>0$
		\[
		\sum_{v\in S}\sum_{\alpha\in A_v}-\log\delta_v(P,\alpha)\le(q_0+q_1+\epsilon)\,\lwh(P)+O(1)
		\]
		for all $P\in\P^1_{\w}(\Z)$ with $\Phi(P)\notin Z(G)$. In particular Conjecture~\ref{conj:integral} holds for every $\alpha$ with $\Phi(\alpha)\in Z(G)$.
	\end{lem}
	
	\begin{proof}
		Put $y=\Phi(P)$ and $S_f=S\setminus\{\infty\}$. For a nonzero integer $m$ one has $\rad(m)\le\prod_{p\in S_f}p\cdot\prod_{v\in S}|m|_v$. Factor $G=g_0\prod_{\beta\in Z(G)}(\beta_1U-\beta_0V)$ over $\bQ$. Since $\max\{|u|_p,|v|_p\}=1$ for finite $p$ and $\max\{|u|,|v|\}=H(y)$, we get
		\[
		\prod_{v\in S}|G(u,v)|_v=C_G\,H(y)^{d}\prod_{v\in S}\prod_{\beta\in Z(G)}\delta_v(y,\beta),
		\]
		where $C_G>0$ does not depend on $y$. Every factor $\delta_v(y,\beta)$ is at most $2$, and for each $v$ the points $\Phi(\alpha)$, $\alpha\in A_v$, are distinct elements of $Z(G)$. Hence $\prod_{v}\prod_{\alpha\in A_v}\delta_v(y,\Phi(\alpha))$ is at least a positive constant times $\prod_v\prod_{\beta}\delta_v(y,\beta)$, and the first inequality follows from Lemma~\ref{lem:chordal}. For the second, let $P$ be integral with coprime representative $\x$. The finitely many $P$ with $x_0x_1=0$ only affect the constant. Otherwise $H(y)=|\x|_{\w}^L$ and $\lwh(P)=\log|\x|_{\w}$, so Conjecture~\ref{conj:radical} and the first inequality give $\prod_{v\in S}\prod_{\alpha\in A_v}\delta_v(P,\alpha)\ge C'|\x|_{\w}^{-(q_0+q_1)-\epsilon}$.
	\end{proof}
	
	We now prove the conjectures under the $abc$ conjecture and show that the exponent cannot be lowered in general. We then examine which weighted forms of $abc$ suffice, prove unconditional analogues for constant targets over characteristic-zero function fields and for almost all real targets, and explain why a general unconditional proof over $\Q$ is out of reach.
	
	\subsection{The conjectures follow from \texorpdfstring{$abc$}{abc}}
	
	The key input is the following consequence of $abc$, which goes back to Elkies \cite{Elkies1991} and appears in various forms in work of Granville \cite{Granville1998} and van Frankenhuijsen \cite{vF2007}; we include the short proof. We use the following prescribed-set form of Belyi's theorem: if $X$ is a smooth projective curve over a number field $k$ and $D$ is a finite set of closed points of $X$, then there is a finite morphism $f:X\to\P^1_k$, ramified only over $\{0,1,\infty\}$, such that $D\subset f^{-1}(\{0,1,\infty\})$; see Belyi~\cite{Belyi1979} and \cite[Thm.~31.1]{Vojta2011}. In particular, when $X=\P^1_{\Q}$ and $D$ is the Galois-stable zero divisor of a form in $\Q[U,V]$, the map can be chosen over $\Q$.
	
	\begin{lem}\label{lem:radical}
		Assume the $abc$ conjecture. Let $\Psi\in\Z[U,V]$ be a binary form of degree $n$ without repeated factors, and let $\epsilon>0$. There is $c>0$ such that
		\[
		\rad\bigl(\Psi(u,v)\bigr)\ge c\max\{|u|,|v|\}^{n-2-\epsilon}
		\]
		for all coprime $(u,v)\in\Z^2$ with $\Psi(u,v)\neq0$.
	\end{lem}
	
	\begin{proof}
		Let $D\subset\P^1(\bQ)$ be the zero set of $\Psi$. By the prescribed-set form of Belyi's theorem \cite[Thm.~31.1]{Vojta2011}, applied to the Galois-stable zero divisor $D$ of $\Psi$, there is a finite morphism $f\colon\P^1_{\Q}\to\P^1_{\Q}$ of some degree $N$, ramified only over $\{0,1,\infty\}$, with $D\subset E:=f^{-1}(\{0,1,\infty\})$. Write $f=[A:C]$ with coprime forms $A,C\in\Z[U,V]$ of degree $N$, and put $B=C-A$. By Riemann--Hurwitz, $E$ has exactly $N+2$ points. Let $R\in\Z[U,V]$ be the product of the distinct primitive irreducible factors of $ABC$, so that $\deg R=N+2$. Since $\Psi$ is squarefree with zeros in $E$, we can write $R=\Psi_0\Psi'$, where $\Psi_0$ is the primitive part of $\Psi$ and $\deg\Psi'=N+2-n$.
		
		The finitely many $y\in E(\Q)\setminus D$ have finitely many coprime representatives, and they can be absorbed into $c$. So let $(u,v)$ be coprime with $y=[u:v]\notin E$, and write $M=\max\{|u|,|v|\}$. There is a nonzero integer $\rho$, depending on the resultant of $A$ and $C$, such that $g:=\gcd(A(u,v),C(u,v))$ divides $\rho$. Moreover $\max\{|A(u,v)|,|C(u,v)|\}\ge c_0M^N$, because $A$ and $C$ have no common zero. Applying $abc$ to the coprime triple $A(u,v)/g+B(u,v)/g=C(u,v)/g$ gives
		\[
		c_0M^N\le|\rho|\,K_\epsilon\,\rad\bigl(A(u,v)B(u,v)C(u,v)\bigr)^{1+\epsilon}.
		\]
		Every prime dividing $ABC(u,v)$ divides the content of $ABC$ or $R(u,v)$, and
		\[
		\rad(R(u,v))\le\rad(\Psi(u,v))\,|\Psi'(u,v)|\le c_1\rad(\Psi(u,v))\,M^{N+2-n}.
		\]
		Hence $M^{N}\ll\bigl(\rad(\Psi(u,v))\,M^{N+2-n}\bigr)^{1+\epsilon}$, which gives $\rad(\Psi(u,v))\gg M^{n-2-\epsilon N}$. Replacing $\epsilon$ by $\epsilon/N$ proves the lemma.
	\end{proof}
	
	\begin{thm}\label{thm:abcgeneral}
		Assume the $abc$ conjecture.
		\begin{enumerate}[(a)]
			\item Let $G\in\Z[U,V]$ be squarefree of degree $d$ with $G(0,1)G(1,0)\neq0$, and let $\epsilon>0$. There is $c>0$ such that
			\[
			\rad\bigl(G(u,v)\bigr)\ge c\,H(\Phi(P))^{d-\epsilon}\,\wh_{\Z}(P)^{-(q_0+q_1)}
			\]
			for all $P\in\P^1_{\w}(\Q)$ with $\Phi(P)=[u:v]$, $u,v$ coprime and $uv\,G(u,v)\neq0$. In particular Conjecture~\ref{conj:radical} holds for all $\w$ and $G$.
			\item Let $S\subset M_{\Q}$ be finite, and for $v\in S$ let $A_v\subset\P^1_{\w}(\bQ)$ be a finite set on which $\Phi$ is injective, with $\Phi(A_v)\cap\{[0:1],[1:0]\}=\emptyset$. Put $B=\bigcup_v\Phi(A_v)$. For every $\epsilon>0$ there is $c'\in\R$ such that
			\[
			\sum_{v\in S}\sum_{\alpha\in A_v}-\log\delta_v(P,\alpha)\le(q_0+q_1+\epsilon)\log\wh_{\Z}(P)+c'
			\]
			for all $P\in\P^1_{\w}(\Q)$ with $\Phi(P)\notin B$. For $P\in\P^1_{\w}(\Z)$ the right-hand side equals $(q_0+q_1+\epsilon)\lwh(P)+c'$; in particular Conjecture~\ref{conj:integral} holds.
		\end{enumerate}
	\end{thm}
	
	\begin{proof}
		(a) Let $G_0$ be the primitive part of $G$. Since $U$ and $V$ do not divide $G$, the form $\Psi=UV\,G_0$ is squarefree of degree $d+2$, and Lemma~\ref{lem:radical} gives $\rad(\Psi(u,v))\ge c\,H(\Phi(P))^{d-\epsilon}$. Let $\x$ be a $\w$-reduced representative of $P$. Then $u\mid x_0^{q_1}$ and $v\mid x_1^{q_0}$, so $\rad(u)\rad(v)\le|x_0x_1|\le\wh_{\Z}(P)^{q_0+q_1}$, and therefore $\rad(\Psi(u,v))\le\wh_{\Z}(P)^{q_0+q_1}\rad(G(u,v))$. For integral $P$ one has $H(\Phi(P))=|\x|_{\w}^L$ and $\wh_{\Z}(P)=|\x|_{\w}$, which gives Conjecture~\ref{conj:radical} after replacing $\epsilon$ by $\epsilon/L$.
		
		(b) Adding the archimedean place with $A_\infty=\emptyset$ changes nothing, so we may assume $\infty\in S$. Let $G$ be the product of the primitive irreducible forms whose zero sets are the Galois orbits of the points of $B$; it is squarefree, and $G(0,1)G(1,0)\neq0$. If $\Phi(P)\in\{[0:1],[1:0]\}$, each $\delta_v(P,\alpha)=\delta_v(\Phi(P),\Phi(\alpha))$ is a fixed positive number, and since $\wh_{\Z}(P)\ge1$ such $P$ only affect the constant. Otherwise $G(u,v)\neq0$, because a rational zero of $G$ would lie in $B$. Lemma~\ref{lem:deduction} and part (a) give
		\[
		\prod_{v}\prod_{\alpha\in A_v}\delta_v(P,\alpha)\ge C\,H(\Phi(P))^{-\epsilon}\,\wh_{\Z}(P)^{-(q_0+q_1)}.
		\]
		Since $H(\Phi(P))=\exp(L\,\lwh(P))\le\wh_{\Z}(P)^L$, replacing $\epsilon$ by $\epsilon/L$ proves the inequality. The last assertion follows from Lemma~\ref{lem:integralpoints}(a).
	\end{proof}
	
	\begin{rem}\label{rem:abcgeneral}
		\begin{enumerate}[(a),leftmargin=*]
			\item The only point where the weights enter is the inequality $\rad(u)\rad(v)\le\wh_{\Z}(P)^{q_0+q_1}$: the numerator and denominator of an integral point are perfect powers, so their radicals are small. Roth's method does not see radicals, and this is why it gives $2L$ instead of $q_0+q_1$.
			\item For $\w=(1,1)$, Theorem~\ref{thm:abcgeneral} is the known implication from $abc$ to Roth's theorem over $\Q$ \cite{Bombieri1994,vF2002}.
			\item Theorem~\ref{thm:abcgeneral} is a special case of the radicalized form of Vojta's conjecture for curves, which follows from $abc$ \cite{vF2007}. For a stacky version of Vojta's conjecture that is equivalent to $abc$, see \cite{NX2025}. The point here is that, for integral points of $\P^1_{\w}$, these inequalities specialize to the explicit exponent $q_0+q_1$, which is sharp by Proposition~\ref{prop:danilov}.
		\end{enumerate}
	\end{rem}
	
	\subsection{A weighted \texorpdfstring{$abc$}{abc} conjecture}
	
	In Theorem~\ref{thm:abcgeneral} the $abc$ conjecture is applied to triples in which two terms are perfect powers. This suggests isolating the corresponding special case.
	
	\begin{defi}\label{def:wabc}
		Let $p,q\ge1$ be integers and $A,B$ nonzero integers. The \emph{weighted $abc$ conjecture} $\mathrm{wABC}(p,q;A,B)$ asserts that for every $\epsilon>0$ there is $K$ such that
		\[
		\max\{|x|^p,|y|^q\}\le K\bigl(|xy|\,\rad(Ax^p+By^q)\bigr)^{1+\epsilon}
		\]
		for all coprime integers $x,y$ with $xy(Ax^p+By^q)\neq0$.
	\end{defi}
	
	Compared with $abc$ for the triple $(Ax^p,By^q,Ax^p+By^q)$, the perfect powers contribute $|x|$ and $|y|$ instead of $\rad(x^p)$ and $\rad(y^q)$, which is the convention for stacky points of orders $p$ and $q$. In the range $|x|^3\asymp|y|^2$, the statement $\mathrm{wABC}(3,2;1,-1)$ is equivalent to the radical form $\rad(x^3-y^2)\gg_\epsilon|x|^{1/2-\epsilon}$ of Hall's conjecture, which Langevin derived from $abc$; see \cite[\S2]{Waldschmidt2004}. A uniform version for perfect powers without radicals, $|x^p-y^q|\gg_\epsilon\max\{x^p,y^q\}^{1-1/p-1/q-\epsilon}$, is stated in \cite[Conj.~2.6]{Waldschmidt2004}, where it is derived from $abc$.
	
	\begin{prop}\label{prop:abcwabc}
		The $abc$ conjecture implies $\mathrm{wABC}(p,q;A,B)$ for all $p,q,A,B$.
	\end{prop}
	
	\begin{proof}
		Let $x,y$ be coprime with $xy\neq0$ and $c=Ax^p+By^q\neq0$, and let $g=\gcd(Ax^p,By^q)$. A prime $\ell\mid g$ divides $AB$, since otherwise it would divide $x$ and $y$. If $\ell\nmid x$ then $v_\ell(g)\le v_\ell(A)$, and if $\ell\nmid y$ then $v_\ell(g)\le v_\ell(B)$; hence $g\mid AB$. Applying $abc$ to $Ax^p/g+By^q/g=c/g$ gives
		\[
		\max\{|x|^p,|y|^q\}\le\max\{|Ax^p|,|By^q|\}\le|AB|\,K_\epsilon\bigl(\rad(AB)\,|xy|\,\rad(c)\bigr)^{1+\epsilon}.\qedhere
		\]
	\end{proof}
	
	A weighted $abc$ conjecture for a single signature suffices for rational targets and, more generally, for targets whose powers are rational.
	
	\begin{prop}\label{prop:kummer}
		Let $b\in\bQ^*$ and $n\ge1$ with $b^n=r_0/r_1$, where $r_0,r_1$ are nonzero coprime integers, and put $G_n=r_1U^n-r_0V^n$. If $\mathrm{wABC}(nq_1,nq_0;r_1,-r_0)$ holds, then Conjecture~\ref{conj:radical} holds for $\w$ and $G_n$. Consequently, the $S$-adic inequality of Lemma~\ref{lem:deduction} and Conjecture~\ref{conj:integral} hold for every $\alpha$ with $\Phi(\alpha)=[\zeta b:1]$, $\zeta^n=1$. For $n=1$ this covers all rational targets other than $[0:1]$ and $[1:0]$.
	\end{prop}
	
	\begin{proof}
		The form $G_n$ has the $n$ distinct zeros $[\zeta b:1]$, so it is squarefree, and $G_n(0,1)G_n(1,0)=-r_0r_1\neq0$. For coprime $\x$ with $x_0x_1G_n(x_0^{q_1},x_1^{q_0})\neq0$, the hypothesis applied to $(x,y)=(x_0,x_1)$ gives
		\[
		|\x|_{\w}^{nL}=\max\{|x_0|^{nq_1},|x_1|^{nq_0}\}\le K\bigl(|x_0x_1|\,\rad(G_n(x_0^{q_1},x_1^{q_0}))\bigr)^{1+\epsilon}.
		\]
		Since $|x_0x_1|\le|\x|_{\w}^{q_0+q_1}$, this yields $\rad(G_n(x_0^{q_1},x_1^{q_0}))\gg|\x|_{\w}^{nL-q_0-q_1-\epsilon nL}$, which is Conjecture~\ref{conj:radical} for $G_n$ after renaming $\epsilon$. The rest follows from Lemma~\ref{lem:deduction}.
	\end{proof}
	
	For a general algebraic target $b$ the proof of Theorem~\ref{thm:abcgeneral} passes through a Belyi map, and its values at integral points are no longer perfect-power triples. The next subsection shows how to restore this structure.
	
	\subsection{Weighted Belyi maps}
	
	\begin{defi}\label{def:wbelyi}
		Let $m_0,m_1\ge2$ be integers and $S\subset\P^1(\bQ)$ a finite set. A \emph{weighted Belyi map of signature $(m_0,m_1)$ for $S$} is a finite morphism $f\colon\P^1_{\Q}\to\P^1_{\Q}$, unramified outside $\{0,1,\infty\}$, such that every point of $f^{-1}(0)$ has ramification index $m_0$, every point of $f^{-1}(1)$ has ramification index $m_1$, and $S\subset f^{-1}(\infty)$.
	\end{defi}
	
	We apply this with $S=\{[0:1],[1:0]\}\cup Z(G)$. The stacky points of $\P^1_{\w}$ are thus placed over the cusp $\infty$, where their weights enter through the inequality $\rad(u)\rad(v)\le|x_0x_1|$. The uniform ramification over $0$ and $1$ makes the values of $f$ at rational points into triples consisting of a cube, a square and a third integer, to which $\mathrm{wABC}(3,2;\cdot,\cdot)$ applies.
	
	\begin{thm}\label{thm:wbelyiexists}
		Let
		\[
		J(\lambda)=\frac{4(\lambda^2-\lambda+1)^3}{27\lambda^2(\lambda-1)^2}.
		\]
		If $g\colon\P^1_{\Q}\to\P^1_{\Q}$ is unramified outside $\{0,1,\infty\}$ and $S\subset g^{-1}(\{0,1,\infty\})$, then $J\circ g$ is a weighted Belyi map of signature $(3,2)$ for $S$. In particular such maps exist for every finite $S\subset\P^1(\bQ)$.
	\end{thm}
	
	\begin{proof}
		Put $\varphi(\lambda)=\lambda^2-\lambda+1$, $\psi(\lambda)=(\lambda+1)(\lambda-2)(2\lambda-1)$ and $\omega(\lambda)=\lambda^2(\lambda-1)^2$. One checks the identity $4\varphi^3-\psi^2=27\omega$, so that $J=4\varphi^3/(27\omega)$ and $J-1=\psi^2/(27\omega)$. Hence $J$ has degree $6$, and its fibres are as follows:
		\begin{itemize}
			\item over $0$, the two zeros of $\varphi$, each with ramification index $3$;
			\item over $1$, the points $-1,\tfrac12,2$, each with index $2$;
			\item over $\infty$, the points $0,1,\infty$, each with index $2$.
		\end{itemize}
		Since $2\cdot6-2=4+3+3$, Riemann--Hurwitz shows that $J$ is unramified elsewhere. Now let $f=J\circ g$. It is unramified outside $J(\{0,1,\infty\})\cup\{0,1,\infty\}=\{0,1,\infty\}$. If $f(z)=0$, then $g(z)$ is a zero of $\varphi$, which does not lie in $\{0,1,\infty\}$, so $g$ is unramified at $z$ and $e_z(f)=3$. Likewise $e_z(f)=2$ for $z\in f^{-1}(1)$. Finally $f^{-1}(\infty)=g^{-1}(\{0,1,\infty\})\supset S$. For existence, replace $S$ by the finite Galois-stable set of all its conjugates. The prescribed-set form of Belyi's theorem \cite[Thm.~31.1]{Vojta2011} gives a map $g$ over $\Q$ with $S\subset g^{-1}(\{0,1,\infty\})$; the composition $J\circ g$ then has all points of $S$ over $\infty$, since $J(0)=J(1)=J(\infty)=\infty$.
	\end{proof}
	
	The function $J$ is $j/1728$ written in terms of the Legendre parameter $\lambda$, so the maps of Theorem~\ref{thm:wbelyiexists} are Belyi maps followed by the $j$-invariant.
	
	\begin{thm}\label{thm:wbelyi}
		Let $G\in\Z[U,V]$ be squarefree of degree $d\ge1$ with $G(0,1)G(1,0)\neq0$. There is a finite set $E_G$ of positive integers, depending only on $G$, with the following property. If $\mathrm{wABC}(3,2;4e,-1)$ holds for every $e\in E_G$, then Conjecture~\ref{conj:radical} holds for $G$ and every coprime weight vector $\w$, and hence so do the conclusions of Lemma~\ref{lem:deduction}.
	\end{thm}
	
	\begin{proof}
		Let $S=\{[0:1],[1:0]\}\cup Z(G)$ and let $g$ be a Belyi map for $S$ of degree $n$, written $g=[G_1:G_2]$ with coprime $G_1,G_2\in\Z[U,V]$ of degree $n$. Let $\rho$ be a nonzero integer such that $\gcd(G_1(u,v),G_2(u,v))$ divides $\rho$ for all coprime $(u,v)$ (for instance the resultant of $G_1$ and $G_2$), and let $E_G$ be the set of positive divisors of $27\rho^3$. Homogenizing the identity $4\varphi^3-\psi^2=27\omega$ gives
		\[
		4X^3-Y^2=27W,
		\]
		where
		\[
		\begin{aligned}
			X&=G_1^2-G_1G_2+G_2^2,\\
			Y&=(G_1+G_2)(G_1-2G_2)(2G_1-G_2),\\
			W&=G_1^2G_2^2(G_1-G_2)^2
		\end{aligned}
		\]
		are forms of degrees $2n$, $3n$ and $6n$. In the notation of Theorem~\ref{thm:wbelyiexists}, $J\circ g=[4X^3:27W]$ and $J\circ g-1=Y^2/(27W)$.
		
		Let $\x$ be coprime with $x_0x_1G(u,v)\neq0$, where $(u,v)=(x_0^{q_1},x_1^{q_0})$, and put $y=[u:v]$ and $H=H(y)=|\x|_{\w}^L$. The rational points $y$ with $g(y)\in\{0,1,\infty,-1,\tfrac12,2\}$ are finite in number, and each comes from at most four coprime pairs $\x$. At these $\rad(G(u,v))\ge1$, so they only affect the constant. For the remaining $\x$, the integers $\tilde x=X(u,v)$, $\tilde y=Y(u,v)$ and $\tilde w=W(u,v)$ are nonzero.
		
		\emph{Step 1: the common factor.} Write $G_1(u,v)=hs'$ and $G_2(u,v)=ht'$ with $h\mid\rho$ and $\gcd(s',t')=1$. Then $\tilde x=h^2X_0$ and $\tilde y=h^3Y_0$ with $X_0=s'^2-s't'+t'^2$ and $Y_0=(s'+t')(s'-2t')(2s'-t')$. The resultant of these two binary forms is $\pm27$, so $\gcd(X_0,Y_0)\mid27$. Hence $e:=\gcd(\tilde x,\tilde y)$ divides $27h^3$, and $e\in E_G$.
		
		\emph{Step 2: the triple.} Write $\tilde x=ex'$ and $\tilde y=ey'$, with $x',y'$ coprime and nonzero. Then
		\[
		4e\,x'^3-y'^2=27\tilde w/e^2=:c'\neq0,
		\]
		and $c'$ divides $27\tilde w$.
		
		\emph{Step 3: size.} The forms $X^3$ and $W$ of degree $6n$ have no common zero on $\P^1(\CC)$: a common zero $z$ would satisfy both $\varphi(g(z))=0$ and $g(z)\in\{0,1,\infty\}$. Hence $\max\{|\tilde x|^3,|\tilde w|\}\gg H^{6n}$. Since $27|\tilde w|\le5\max\{|\tilde x|^3,|\tilde y|^2\}$, we get $\max\{|x'|^3,|y'|^2\}\gg H^{6n}$. On the other hand $|x'|\le|\tilde x|\ll H^{2n}$ and $|y'|\le|\tilde y|\ll H^{3n}$.
		
		\emph{Step 4: the radical.} The zero set of $G_1G_2(G_1-G_2)$ is $g^{-1}(\{0,1,\infty\})$, which has $n+2$ points by Riemann--Hurwitz and contains $S$. Let $T$ be the product of the distinct primitive irreducible factors of $G_1G_2(G_1-G_2)$. Then $\deg T=n+2$ and $T=\pm UV\,G_0\,T'$, where $G_0$ is the primitive part of $G$ and $\deg T'=n-d$. Every prime dividing $\tilde w$ divides the content of $G_1G_2(G_1-G_2)$ or $T(u,v)$. Therefore
		\[
		\begin{aligned}
			\rad(c')
			&\le3\rad(\tilde w)
			\ll\rad(u)\rad(v)\rad(G(u,v))\,|T'(u,v)|\\
			&\ll|x_0x_1|\,\rad(G(u,v))\,H^{n-d}.
		\end{aligned}
		\]
		
		\emph{Step 5: conclusion.} Applying $\mathrm{wABC}(3,2;4e,-1)$ to $(x',y')$ and using Steps 3 and 4, we obtain
		\[
		H^{6n}\ll\bigl(H^{2n}H^{3n}\,|x_0x_1|\,\rad(G(u,v))\,H^{n-d}\bigr)^{1+\epsilon}.
		\]
		Using $|x_0x_1|\le|\x|_{\w}^{q_0+q_1}$ and $H=|\x|_{\w}^L$, this gives
		\[
		\rad(G(u,v))\gg H^{d-6n\epsilon}|\x|_{\w}^{-(q_0+q_1)}=|\x|_{\w}^{dL-q_0-q_1-6nL\epsilon}.
		\]
		Replacing $\epsilon$ by $\epsilon/(6nL)$ completes the proof. Note that $E_G$ depends only on $g$, and hence only on $G$, and not on $\w$.
	\end{proof}
	
	\begin{prop}\label{prop:equiv}
		The $abc$ conjecture holds if and only if\linebreak $\mathrm{wABC}(3,2;4e,-1)$ holds for each
		\[
		e\in\{1,3,9,27\}.
		\]
	\end{prop}
	
	\begin{proof}
		One direction is Proposition~\ref{prop:abcwabc}. Conversely, let $a+b=c$ with $a,b,c$ coprime and $abc\neq0$, and put $M=\max\{|a|,|b|,|c|\}$. With $X,Y,W$ as in the proof of Theorem~\ref{thm:wbelyi}, taking $G_1=U$ and $G_2=V$ and evaluating at $(a,c)$, we get
		\[
		\tilde x=a^2-ac+c^2,\qquad \tilde y=(a+c)(a-2c)(2a-c),\qquad 4\tilde x^3-\tilde y^2=27a^2b^2c^2.
		\]
		If $\tilde y=0$, then $a=-c$, $a=2c$ or $c=2a$, and coprimality forces $M\le2$. Otherwise, since $\gcd(a,c)=1$, the integer $e=\gcd(\tilde x,\tilde y)$ divides $27$. Write $\tilde x=ex'$ and $\tilde y=ey'$. Then $2\tilde x=a^2+c^2+(a-c)^2\ge\max\{a^2,c^2\}$, and $M\le2\max\{|a|,|c|\}$, so $\tilde x\ge M^2/8$. Also $|\tilde y|\le18M^3$ and $\rad(4ex'^3-y'^2)\le3\rad(abc)$. Now $\mathrm{wABC}(3,2;4e,-1)$ gives
		\[
		M^6\ll\bigl(M^5\rad(abc)\bigr)^{1+\epsilon},
		\]
		Thus
		\[
		\rad(abc)\gg M^{6/(1+\epsilon)-5}=M^{(1-5\epsilon)/(1+\epsilon)}.
		\]
		Given $\eta>0$, choose $\epsilon>0$ so small that $(1-5\epsilon)/(1+\epsilon)\ge1-\eta$. Then $M\ll_\eta\rad(abc)^{1/(1-\eta)}$, which is the usual $abc$ inequality after the standard reparametrization of the exponent.
	\end{proof}
	
	The integers $\tilde x,\tilde y$ are, up to the factors $16$ and $\pm32$, the invariants $c_4,c_6$ of the Frey curve $y^2=x(x-a)(x+b)$, whose discriminant is $16a^2b^2c^2$. The argument above is therefore a form of the familiar link between $abc$ and Szpiro-type inequalities for Frey curves; see \cite{Oesterle1988}.
	
	\begin{cor}\label{cor:universal}
		Let $\w=(2,3)$. If Conjecture~\ref{conj:radical} holds for $\w$ and each of the four linear forms $4eU-V$, $e\in\{1,3,9,27\}$, then the $abc$ conjecture holds. Consequently Conjectures~\ref{conj:integral} and~\ref{conj:radical} hold for all coprime weight vectors, all forms $G$ and all algebraic targets.
	\end{cor}
	
	\begin{proof}
		For $\w=(2,3)$ we have $L=6$ and $q_0+q_1=5$. For $G=4eU-V$, Conjecture~\ref{conj:radical} states that $\rad(4e\,x_0^3-x_1^2)\gg|\x|_{\w}^{1-\epsilon}$ for coprime $\x$ with $x_0x_1(4ex_0^3-x_1^2)\neq0$. Apply this to $(x_0,x_1)=(x',y')$ from the proof of Proposition~\ref{prop:equiv}. There $|\x|_{\w}\ge|x'|^{1/2}\ge(\tilde x/27)^{1/2}\gg M$ and $\rad(4ex'^3-y'^2)\le3\rad(abc)$, so $\rad(abc)\gg M^{1-\epsilon}$. The last assertion follows from Theorem~\ref{thm:abcgeneral} and Lemma~\ref{lem:deduction}.
	\end{proof}
	
	\begin{rem}\label{rem:wbelyi}
		\begin{enumerate}[(a),leftmargin=*]
			\item Theorem~\ref{thm:wbelyi} and Proposition~\ref{prop:equiv} show that weighted Belyi maps allow Theorem~\ref{thm:abcgeneral}(a) to be derived from the weighted $abc$ conjecture of signature $(3,2)$ alone, but this does not weaken the hypothesis, since that conjecture is equivalent to $abc$. Corollary~\ref{cor:universal} shows that the radical form of the conjecture for $\w=(2,3)$ and four rational targets already implies all other cases.
			\item The radical is essential in this argument. For the pairs $(x',y')$ used above, $|4ex'^3-y'^2|=27a^2b^2c^2/e^2\gg M^4$ while $|\x|_{\w}\ll M$. So the form of Conjecture~\ref{conj:integral} with $|F_\alpha(\x)|$ in place of the radical is automatically satisfied there and gives no information.
			\item Suppose that for some signature $(m_0,m_1)$ there is a weighted Belyi map $f$ of signature $(m_0,m_1)$ for $\{0,1,\infty\}$. The argument of Proposition~\ref{prop:equiv} then shows that $\mathrm{wABC}(m_0,m_1;\cdot,\cdot)$, for finitely many coefficient pairs, implies $abc$. The exponent count uses Riemann--Hurwitz in the form $N/m_0+N/m_1+\#f^{-1}(\infty)=N+2$, where $N=\deg f$. We have not investigated which signatures admit such maps.
			\item In Proposition~\ref{prop:kummer} the power map $t\mapsto t^n$ plays the role of the Belyi map. It respects the stacky structure, which is why a single weighted $abc$ statement suffices there.
		\end{enumerate}
	\end{rem}
	
	\subsection{Sharpness}
	
	\begin{prop}\label{prop:danilov}
		Let $\w=(2,3)$, so $L=6$ and $q_0+q_1=5$. There are infinitely many coprime $\x\in\Z^2$ with $0<|x_0^3-x_1^2|<|\x|_{\w}$. Hence for $\Phi(\alpha)=[1:1]$ the constant $q_0+q_1$ in Conjecture~\ref{conj:integral} and Theorem~\ref{thm:abcgeneral} cannot be lowered.
	\end{prop}
	
	\begin{proof}
		For $t\in\Z$ put $x_0=5^5t^2-3000t+719$ and $w(t)=5^3t^2-114t+26$. One checks the identity
		\[
		x_0^3-w(t)\bigl(5^6t^2-5^3\cdot123\,t+3781\bigr)^2=27(2t-1)
		\]
		due to Danilov and Elkies \cite{Danilov1982,Elkies2000}. Since $125\,w(t)=(125t-57)^2+1$, the values of $t$ with $w(t)=s^2$ a square correspond to solutions of $X^2-125s^2=-1$ with $X\equiv-57\pmod{125}$. One solution is $(X,s)=(-682,61)$, coming from $t=-5$. Multiplying $-682+61\sqrt{125}$ by powers of a unit $a+b\sqrt{125}$ of norm $1$ with $a\equiv1$ and $b\equiv0\pmod{125}$ preserves the congruence, so there are infinitely many such $t$. For these put $x_1=s(5^6t^2-5^3\cdot123\,t+3781)$. Then $x_0^3-x_1^2=27(2t-1)\neq0$. Since $4x_0\equiv1\pmod{2t-1}$ and $x_0\equiv-(t^2+1)\not\equiv0\pmod3$, we have $\gcd(x_0,27(2t-1))=1$, and hence $\gcd(x_0,x_1)=1$. The quadratic $5^5t^2-3000t+719$ is negative only for $t$ in an interval inside $(0.46,0.5)$, which contains no integer; hence $x_0>0$ and $|\x|_{\w}\ge x_0^{1/2}$. Along the Pell sequence one has $|t|\to\infty$, hence $x_0\to\infty$ and $\lwh([x_0:x_1]_{\w})\to\infty$. Finally $27|2t-1|/x_0^{1/2}\to54/5^{5/2}<1$ as $|t|\to\infty$. For $\alpha=[1:1]_{\w}$ this gives $-\log\delta_\infty(P,\alpha)=-\log|x_0^3-x_1^2|+6\log|\x|_{\w}\ge5\,\lwh(P)$.
	\end{proof}
	
	For $\w=(2,3)$ and $\Phi(\alpha)=[1:1]$, Conjecture~\ref{conj:integral} is the coprime case of Hall's conjecture $|x^3-y^2|\gg_\epsilon|x|^{1/2-\epsilon}$ \cite{Hall1971}, and Proposition~\ref{prop:danilov} is Danilov's theorem. More generally, Theorem~\ref{thm:abcgeneral} bounds the coprime solutions of $\beta_1x^{q_1}-\beta_0y^{q_0}=m$ by $|(x,y)|_{\w}\ll_\epsilon|m|^{1/(L-q_0-q_1)+\epsilon}$ whenever $L>q_0+q_1$.
	
	\subsection{Unconditional results}
	
	Over one-variable function fields, the Mason--Stothers theorem \cite{Mason1984,Stothers1981} gives an unconditional analogue of the expected exponent for constant targets and polynomial approximants. The following proposition is the precise statement used here. Let $\mathbf k$ be an algebraically closed field of characteristic $0$. For $x_0,x_1\in\mathbf k[t]$ put $d_{\w}(\x)=\max\{\deg x_0/q_0,\deg x_1/q_1\}$, the analogue of $\log|\x|_{\w}$ for the place at infinity.
	
	\begin{prop}\label{prop:functionfield}
		Let $x_0,x_1\in\mathbf k[t]$ be coprime and not both constant, and let $b\in\mathbf k^*$. If $F=x_0^{q_1}-b\,x_1^{q_0}\neq0$, then
		\[
		\deg F\ge(L-q_0-q_1)\,d_{\w}(\x)+1.
		\]
		For $\w=(2,3)$ and $b=1$ the bound is attained, for example by $(x_0,x_1)=(t^2+2,\,t^3+3t)$, and more generally by the pairs of Stothers and Zannier \cite{Stothers1981,Zannier1995}.
	\end{prop}
	
	\begin{proof}
		The polynomials $A=x_0^{q_1}$, $B=-b\,x_1^{q_0}$ and $F=A+B$ are pairwise coprime and not all constant. By Mason--Stothers, $\max\{\deg A,\deg B,\deg F\}\le n_0(ABF)-1$, where $n_0$ counts distinct roots. The left side is at least $\max\{q_1\deg x_0,q_0\deg x_1\}=L\,d_{\w}(\x)$. The right side is at most $\deg x_0+\deg x_1+\deg F-1\le(q_0+q_1)d_{\w}(\x)+\deg F-1$. For the example, $(t^2+2)^3-(t^3+3t)^2=3t^2+8$ has degree $2=(6-5)\cdot1+1$. Zannier shows that for every $\delta\ge1$ there are $f,g$ with $\deg f=2\delta$, $\deg g=3\delta$ and $\deg(f^3-g^2)=\delta+1$. Such pairs are coprime: if $h=\gcd(f,g)\neq1$, apply Mason--Stothers to $f^3/e$, $-g^2/e$ and $(f^3-g^2)/e$ with $e=\gcd(f^3,g^2)$. Using $n_0(fg)\le\deg f+\deg g-1$, this gives $\deg(f^3-g^2)\ge\delta+2$.
	\end{proof}
	
	Over $\Q$, the conjecture holds for almost all targets.
	
	\begin{prop}\label{prop:metric}
		For Lebesgue-almost every $b\in\R$ and every $\epsilon>0$ there is $C=C(b,\epsilon)>0$ such that
		\[
		|x_0^{q_1}-b\,x_1^{q_0}|\ge C\,|\x|_{\w}^{L-q_0-q_1-\epsilon}
		\]
		for all coprime $\x\in\Z^2$.
	\end{prop}
	
	\begin{proof}
		It suffices to treat $b$ in a compact interval $I\subset\R\setminus\{0\}$ and a fixed $\epsilon$. For coprime $\x$ let $E_{\x}=\{b\in I:|x_0^{q_1}-b\,x_1^{q_0}|<|\x|_{\w}^{L-q_0-q_1-\epsilon}\}$, and let $T\le|\x|_{\w}<2T$ with $T$ large. If $E_{\x}\neq\emptyset$, then $|x_1|^{q_0}\gg T^L$. Indeed $\max\{|x_0|^{q_1},|x_1|^{q_0}\}=|\x|_{\w}^L$, and $|x_0^{q_1}-bx_1^{q_0}|$ is $o(T^L)$ while $b$ is bounded away from $0$ and $\infty$. Hence $E_{\x}$ is an interval of length $\ll T^{-q_0-q_1-\epsilon}$. There are $\ll T^{q_0+q_1}$ vectors with $|\x|_{\w}<2T$, so $\sum_{T\le|\x|_{\w}<2T}|E_{\x}|\ll T^{-\epsilon}$, which is summable over $T=2^k$. By the Borel--Cantelli lemma, for this fixed $\epsilon$ almost every $b\in I$ lies in only finitely many $E_{\x}$. Intersecting the resulting full-measure sets for $\epsilon=1/n$, $n\ge1$, gives one full-measure set on which the asserted conclusion holds for every $\epsilon>0$. The vectors with $x_1=0$ form a finite coprime exceptional set and may be absorbed into the constant. For irrational $b$ the finitely many remaining exceptional $\x$ have $x_0^{q_1}-bx_1^{q_0}\neq0$, so they can be absorbed into $C$.
	\end{proof}
	
	\begin{rem}\label{rem:barrier}
		An unconditional proof of Conjecture~\ref{conj:integral} over $\Q$ would in particular establish the coprime case of Hall's conjecture. The best unconditional results in that direction come from Baker's method and are of logarithmic strength \cite{Baker1968}. There is also a structural reason why the methods behind Theorem~\ref{thm:wroth} cannot reach the conjecture. The map $\Phi$ identifies $\P^1_{\w}$ with $\P^1$ and maps $\P^1_{\w}(\Q)$ onto $\P^1(\Q)$, and on rational points the exponent $2L$ is sharp (Proposition~\ref{prop:optimal}). Any proof of Conjecture~\ref{conj:integral} must therefore use integrality, and in the proof of Theorem~\ref{thm:abcgeneral} integrality enters only through radicals. We do not know how to use information of this kind in the auxiliary polynomial method or in the subspace theorem.
	\end{rem}
	
	\section{A subspace theorem for diagonal hypersurfaces}\label{sec:subspace}
	
	We return to arbitrary weights $\w=(q_0,\dots,q_n)$. Vectors $\cc_1,\dots,\cc_t\in\bQ^{n+1}\setminus\{0\}$ are in \emph{general position} if any $\min\{t,n+1\}$ of them are linearly independent; equivalently, the hyperplanes $\Lambda_{\cc_j}$ are in general position \cite[Def.~7.1]{Vojta2011}.
	
	\begin{thm}[Schmidt]\label{thm:schmidt}
		Let $S\subset M_k$ be finite and contain the archimedean places, and let $\cc_1,\dots,\cc_t\in\bQ^{n+1}$ be in general position. For every $\epsilon>0$ and $c\in\R$,
		\[
		\sum_{j=1}^t\sum_{v\in S}\lambda_{\cc_j,v}(y)\le(n+1+\epsilon)\,h_k(y)+c
		\]
		for all $y\in\P^n(k)$ outside a finite union of proper linear subspaces of $\P^n_k$.
	\end{thm}
	
	For hyperplanes defined over $k$ this is \cite[Thm.~7.11]{Vojta2011}; for algebraic coefficients see \cite{Schmidt1980,Schmidt1991}, the remarks after \cite[Thm.~7.11]{Vojta2011}, and \cite[Ch.~7]{BG2006}.
	
	\begin{thm}[Pullback of the subspace theorem]\label{thm:wsubspace}
		With $S$, $\cc_1,\dots,\cc_t$, $\epsilon$, $c$ as in Theorem~\ref{thm:schmidt}, there are proper linear subspaces $W_1,\dots,W_u\subset\P^n_k$ such that
		\[
		\sum_{j=1}^t\sum_{v\in S}\zeta_{\cc_j,v}(P)\le(n+1+\epsilon)\,L\,\lwh_k(P)+c
		\]
		for every $P\in\P^n_{\w}(k)\setminus\bigcup_jH_{\cc_j}$ with $\Phi(P)\notin W_1\cup\dots\cup W_u$. The exceptional points lie on finitely many diagonal hypersurfaces $H_{\cc}$, $\cc\in k^{n+1}\setminus\{0\}$.
	\end{thm}
	
	\begin{proof}
		Apply Theorem~\ref{thm:schmidt} to $y=\Phi(P)$, using Lemmas~\ref{lem:weilcompare} and~\ref{lem:veronese}; every proper linear subspace lies in some $\Lambda_{\cc}$, and $\Phi^{-1}(\Lambda_{\cc})=H_{\cc}$.
	\end{proof}
	
	When $\cc_j\in k^{n+1}$, Proposition~\ref{prop:fmt} turns the conclusion into $\sum_jN_S(\cc_j,P)\ge(t-n-1-\epsilon)L\,\lwh_k(P)-c'$.
	
	\begin{cor}\label{cor:product}
		Let $\cc_1,\dots,\cc_t\in\bQ^{n+1}$ be in general position and $\epsilon>0$. There are $C>0$ and proper linear subspaces $W_1,\dots,W_u$ of $\P^n_{\Q}$ such that
		\[
		\prod_{j=1}^t|F_{\cc_j}(\x)|\ge C\,|\x|_{\w}^{L(t-n-1-\epsilon)}
		\]
		for every $\w$-reduced $\x\in\Z^{n+1}$ with $\Phi([\x]_{\w})\notin\bigcup_iW_i$ and $\prod_jF_{\cc_j}(\x)\neq0$.
	\end{cor}
	
	\begin{proof}
		Take $k=\Q$, $S=\{\infty\}$, $c=0$ in Theorem~\ref{thm:wsubspace}. Since
		\[
		\zeta_{\cc_j,\infty}([\x]_{\w})=-\log|F_{\cc_j}(\x)|+L\log|\x|_{\w}+\log\nm{\cc_j}_\infty
		\]
		and $\lwh([\x]_{\w})\le\log|\x|_{\w}$ (Lemma~\ref{lem:integral}), the claim holds with $C=\prod_j\nm{\cc_j}_\infty$.
	\end{proof}
	
	\begin{rem}\label{rem:subspace}
		For $t=n+1$ Corollary~\ref{cor:product} reads $\prod_j|F_{\cc_j}(\x)|\ge C|\x|_{\w}^{-\epsilon L}$. As in the classical case it has content only when some $\cc_j$ has irrational coordinates; for rational $\cc_j$, Theorem~\ref{thm:wsubspace} is substantial once $S$ contains finite places. The version of Schmidt's theorem with hyperplanes depending on $v$ \cite[Thm.~7.12]{Vojta2011} transfers in the same way. For general weighted hypersurfaces one must embed $\P^n_{\w}$ by a multiple of $\cO(L)$ and use the methods of Corvaja--Zannier~\cite{CZ2004} and Evertse--Ferretti~\cite{EF2008}; see \cite[\S\S20--21]{Vojta2011}.
	\end{rem}
	
	\section{Weighted binary forms and Thue equations}\label{sec:thue}
	
	The classical finiteness theory of binary form equations begins with Thue's theorem \cite{Thue1909}. Here we record the consequence of Roth's theorem that is most natural for weighted homogeneous binary forms. In this section $\w=(q_0,q_1)$ with $\gcd(q_0,q_1)=1$. A polynomial $F\in\bQ[X,Y]$ is \emph{weighted homogeneous of weighted degree $D$} if $q_0i+q_1j=D$ for every monomial $X^iY^j$ of $F$.
	
	\begin{lem}\label{lem:structure}
		Every nonzero weighted homogeneous $F$ of weighted degree $D$ can be written uniquely as
		\[
		F(X,Y)=X^aY^b\,G(X^{q_1},Y^{q_0})
		\]
		with $0\le a<q_1$, $0\le b<q_0$ and $G\in\bQ[U,V]$ a binary form of degree $\kappa\ge0$; moreover $D=aq_0+bq_1+\kappa L$.
	\end{lem}
	
	\begin{proof}
		If $q_0i+q_1j=D$ then $i\equiv Dq_0^{-1}\pmod{q_1}$ and $j\equiv Dq_1^{-1}\pmod{q_0}$, so $a,b$ depend only on $D$. Writing $i=a+q_1s$ and $j=b+q_0s'$, we get $s+s'=(D-aq_0-bq_1)/L=:\kappa$ for every monomial.
	\end{proof}
	
	\begin{lem}[Thue--Siegel--Roth]\label{lem:thueroth}
		Let $G\in\bQ[U,V]$ be a binary form of degree $\kappa\ge0$ without repeated factors, and let $\epsilon>0$. There is $C>0$ with $|G(u,v)|\ge C\max\{|u|,|v|\}^{\kappa-2-\epsilon}$ for all coprime $(u,v)\in\Z^2$ with $G(u,v)\neq0$.
	\end{lem}
	
	\begin{proof}
		Let $\kappa\ge1$ and $G=g\prod_{i=1}^\kappa(\theta_{i1}U-\theta_{i0}V)$ with distinct $\theta_i\in\P^1(\bQ)$. For coprime $(u,v)$ and $y=[u:v]$ we have $h(y)=\log\max\{|u|,|v|\}$ and
		\[
		-\log|G(u,v)|=\sum_{i=1}^\kappa\lambda_{\theta_i,\infty}(y)-\kappa\,h(y)+c_G.
		\]
		For $i\neq j$ the function $y\mapsto\max\{\delta_\infty(y,\theta_i),\delta_\infty(y,\theta_j)\}$ is continuous and positive on the compact space $\P^1(\CC)$. Hence there is $\eta>0$ such that at most one $i$ has $\delta_\infty(y,\theta_i)<\eta$, and so $\sum_i\lambda_{\theta_i,\infty}(y)\le\max_i\lambda_{\theta_i,\infty}(y)+(\kappa-1)\log\eta^{-1}$. By Theorem~\ref{thm:roth}, applied to each $\theta_i$, $\max_i\lambda_{\theta_i,\infty}(y)\le(2+\epsilon)h(y)$ outside a finite set of $y$. Each exceptional $y$ has only the representatives $\pm(u,v)$, and there $G(u,v)\neq0$, so the constant can be adjusted.
	\end{proof}
	
	For $(x,y)\in\Z^2$ we write $|(x,y)|_{\w}=\max\{|x|^{1/q_0},|y|^{1/q_1}\}$.
	
	\begin{thm}[Weighted Thue inequality]\label{thm:wthue}
		Let $F\in\bQ[X,Y]$ be weighted homogeneous, written as in Lemma~\ref{lem:structure} with $G$ free of repeated factors. For every $\epsilon>0$ there is $C>0$ such that
		\[
		|F(x,y)|\ge C\,|(x,y)|_{\w}^{L(\kappa-2-\epsilon)}
		\]
		for all coprime $(x,y)\in\Z^2$ with $F(x,y)\neq0$.
	\end{thm}
	
	\begin{proof}
		The pairs with $xy=0$ are $(\pm1,0),(0,\pm1)$ and are absorbed into $C$. For $xy\neq0$ put $u=x^{q_1}$, $v=y^{q_0}$. These are coprime, $\max\{|u|,|v|\}=|(x,y)|_{\w}^L$, and $|F(x,y)|\ge|G(u,v)|$ since $|x^ay^b|\ge1$; apply Lemma~\ref{lem:thueroth}.
	\end{proof}
	
	For $F\in\bQ[X^{q_1},Y^{q_0}]$ we have $\kappa L=D$, and the bound reads
	\[
	|F(x,y)|\ge C\,|(x,y)|_{\w}^{D-(2+\epsilon)L},
	\]
	the weighted form of $|F(x,y)|\gg\max\{|x|,|y|\}^{D-2-\epsilon}$.
	
	\begin{cor}[Weighted Thue equations]\label{cor:wthueeq}
		If $\kappa\ge3$ in Theorem~\ref{thm:wthue} and $m\in\CC^*$, then $F(x,y)=m$ has only finitely many solutions $(x,y)\in\Z^2$.
	\end{cor}
	
	\begin{proof}
		If $x=0$ and $a>0$, then $F(0,y)=0$ and there is no solution because $m\ne0$; if $a=0$, then $F(0,y)$ is a nonzero constant multiple of a power of $y$, so only finitely many $y$ occur. The case $y=0$ is identical, with $b$ in place of $a$. Let $xy\neq0$, put $u=x^{q_1}$, $v=y^{q_0}$, $e=\gcd(u,v)$ and $(u,v)=e(u',v')$. Then $|m|\ge e^\kappa|G(u',v')|$ with $G(u',v')\neq0$. Lemma~\ref{lem:thueroth} with $\epsilon=1/2$ gives $e^\kappa\max\{|u'|,|v'|\}^{\kappa-5/2}\le|m|/C$, which bounds $e$, $|u'|$, $|v'|$, and therefore $|x|$ and $|y|$.
	\end{proof}
	
	\begin{rem}\label{rem:sharpthue}
		\begin{enumerate}[(a),leftmargin=*]
			\item The hypothesis $\kappa\ge3$ cannot be weakened to $\kappa\ge2$ in general. For $\w=(1,1)$ the Pell equation $u^2-2v^2=1$ has $\kappa=2$ and infinitely many solutions. For $\w=(1,2)$ the form $X^2-Y$ has $\kappa=1$, and $x^2-y=1$ has the coprime solutions $(x,x^2-1)$.
			\item For coprime $p,q\ge2$ and $\w=(q,p)$, the form $X^p-Y^q$ has $\kappa=1$, so Theorem~\ref{thm:wthue} alone gives no finiteness. Finiteness of $x^p-y^q=m$ follows from Siegel's theorem \cite{Siegel1929}, \cite[\S D.9]{HS2000}, as the curve has genus $(p-1)(q-1)/2\ge1$.
		\end{enumerate}
	\end{rem}
	
	Under $abc$ the exponent in Theorem~\ref{thm:wthue} improves, and finiteness extends below $\kappa=3$.
	
	\begin{cor}\label{cor:abcthue}
		Assume the $abc$ conjecture. Let $F$ be as in Theorem~\ref{thm:wthue}, and assume that $G(0,1)$ and $G(1,0)$ are nonzero. For every $\epsilon>0$,
		\[
		|F(x,y)|\gg_\epsilon|(x,y)|_{\w}^{\kappa L-q_0-q_1-\epsilon}
		\]
		for all coprime $(x,y)\in\Z^2$ with $F(x,y)\neq0$. In particular, if $\kappa L>q_0+q_1$, then $F(x,y)=m$ has only finitely many coprime solutions for every $m\neq0$.
	\end{cor}
	
	\begin{proof}
		The four pairs with $xy=0$ are absorbed into the constant. For $xy\neq0$, put $P=[x:y]_{\w}\in\P^1_{\w}(\Z)$ and $(u,v)=(x^{q_1},y^{q_0})$, and let $\theta_1,\dots,\theta_\kappa$ be the zeros of $G$, none of which is $[0:1]$ or $[1:0]$. As in the proof of Lemma~\ref{lem:thueroth}, $-\log|G(u,v)|=\sum_i-\log\delta_\infty(\Phi(P),\theta_i)-\kappa L\,\lwh(P)+O(1)$. Theorem~\ref{thm:abcgeneral}, with $S=\{\infty\}$ and $A_\infty$ consisting of one preimage of each $\theta_i$, bounds the sum by $(q_0+q_1+\epsilon)\lwh(P)+O(1)$. Since $|F(x,y)|\ge|G(u,v)|$ and $\lwh(P)=\log|(x,y)|_{\w}$, the inequality follows.
	\end{proof}
	
	For $\kappa=1$ the condition $\kappa L>q_0+q_1$ reads $\frac1{q_0}+\frac1{q_1}<1$, that is, $q_0,q_1\ge2$. For the forms $X^p-Y^q$ of Remark~\ref{rem:sharpthue}(b) this is exactly the case in which the curve $x^p-y^q=m$ has positive genus, in accordance with Siegel's theorem.
	
	\section{Failure of a naive weighted index rescaling}\label{sec:direct}
	
	For $P\in\Z[x_1,\dots,x_m]$ put $P_{\mathbf i}=(i_1!\cdots i_m!)^{-1}\partial^{i_1+\dots+i_m}P/\partial x_1^{i_1}\cdots\partial x_m^{i_m}$. Given weights $q_\ell$ and positive integers $d_\ell$, the natural weighted index of $P$ at $\mathbf z$ is
	\[
	t_{\w}(P;\mathbf z;\mathbf d)=\min\Bigl\{\sum_{\ell=1}^m\frac{q_\ell i_\ell}{d_\ell}:P_{\mathbf i}(\mathbf z)\neq0\Bigr\},
	\]
	which is the classical index of Roth's method \cite[\S D.3]{HS2000} with respect to the parameters $d_\ell/q_\ell$. If $q_\ell\mid d_\ell$ for all $\ell$, it is the usual index for the integers $d_\ell/q_\ell$, and the auxiliary polynomial and Roth's lemma are the classical ones. Without this divisibility, however, the most direct rescaling of the classical degree box can fail, as the following example shows. For $\mathbf z\in K^m$ put $\ind_{\mathbf z}(P)=\min\{i_1+\dots+i_m:P_{\mathbf i}(\mathbf z)\neq0\}$, and recall that $\ind_{\mathbf z}(A+B)=\min\{\ind_{\mathbf z}A,\ind_{\mathbf z}B\}$ when the two differ.
	
	\begin{lem}\label{lem:twopoints}
		Let $P\in K[x_1,\dots,x_m]$ be nonzero of degree at most $1$ in each variable, and let $\mathbf a,\mathbf b\in K^m$ with $a_\ell\neq b_\ell$ for all $\ell$. Then $\ind_{\mathbf a}(P)+\ind_{\mathbf b}(P)\le m$.
	\end{lem}
	
	\begin{proof}
		Induction on $m$, the case $m=0$ being trivial. Write $P=(x_1-a_1)A+B=(x_1-b_1)A+B'$ with $A,B\in K[x_2,\dots,x_m]$ and $B'=B+(b_1-a_1)A$, and let $\mathbf a',\mathbf b'$ be $\mathbf a,\mathbf b$ without their first coordinates. Then $\ind_{\mathbf a}P=\min\{1+\ind_{\mathbf a'}A,\ind_{\mathbf a'}B\}$ and $\ind_{\mathbf b}P=\min\{1+\ind_{\mathbf b'}A,\ind_{\mathbf b'}B'\}$. If $A=0$ we conclude by induction. Otherwise put $s=\ind_{\mathbf a'}A$ and $s'=\ind_{\mathbf b'}A$, so $s+s'\le m-1$. The claim is clear unless $\ind_{\mathbf a}P=s+1$, $\ind_{\mathbf b}P=s'+1$ and $s+s'=m-1$. In that case $\ind_{\mathbf a'}B>s$ and $\ind_{\mathbf b'}B'>s'$. The latter and $B=B'-(b_1-a_1)A$ give $\ind_{\mathbf b'}B=s'$, so $B\neq0$, and induction yields $\ind_{\mathbf a'}B\le m-1-s'=s$, a contradiction.
	\end{proof}
	
	\begin{prop}\label{prop:index}
		Let $m=9$, $\alpha=\sqrt2$, $\epsilon=1/3$, $q_\ell=3$ and $d_\ell=5$ for all $\ell$, so that $m=[\Q(\alpha):\Q]/(2\epsilon^2)$. There is no nonzero $P\in\Z[x_1,\dots,x_9]$ with $\deg_{x_\ell}P\le\lfloor d_\ell/q_\ell\rfloor$ for all $\ell$ and $t_{\w}(P;(\alpha,\dots,\alpha);\mathbf d)\ge m(1-\epsilon)/2$.
	\end{prop}
	
	\begin{proof}
		Here $\lfloor d_\ell/q_\ell\rfloor=1$ and $m(1-\epsilon)/2=3$, and $t_{\w}=\frac35\ind_{\mathbf a}(P)$ with $\mathbf a=(\sqrt2,\dots,\sqrt2)$. The condition therefore forces $\ind_{\mathbf a}(P)\ge5$. As $P$ has rational coefficients, conjugation gives $\ind_{\mathbf b}(P)\ge5$ at $\mathbf b=-\mathbf a$, contradicting Lemma~\ref{lem:twopoints}.
	\end{proof}
	
	By contrast, the classical requirement $\ind_{\mathbf a}(P)\ge3$ under the same degree bounds imposes $92$ rational linear conditions on $512$ coefficients, so it can be met.
	
	\subsection*{Data availability}
	No data were generated or analysed in this work.

\end{document}